\documentclass{amsart}
\usepackage[utf8]{inputenc}     % for éô
\usepackage[english]{babel}     % for proper word breaking at line ends
\usepackage[a4paper, left=1.5in, right=1.5in, top=1.5in, bottom=1.5in]{geometry}
\usepackage{graphicx}           % for ?
\graphicspath{{images/}}

\usepackage[hidelinks]{hyperref}           % page numbers and '\ref's become clickable
\usepackage{doi}						   % makes DOIs clickable in references
\usepackage{caption}
\usepackage{adjustbox}
\usepackage{verbatim}
\usepackage[titletoc]{appendix}

\usepackage{tensor}
\usepackage{accents}

\usepackage{multirow}

\usepackage{amsmath,amssymb}    % for better equations
\usepackage{amsthm}             % for better theorem styles
\usepackage{mathtools}          % for greek math symbol formatting
\usepackage{mathrsfs,color} 	%a4wide
\usepackage[shortlabels]{enumitem}           % for control of 'enumerate' numbering
\usepackage{listings}           % for control of 'itemize' spacing
\setlist[enumerate, 1]{label = \roman*), ref = \roman*),font = \textup}

\usepackage[foot]{amsaddr}
\usepackage{orcidlink}

\title[A Pontryagin class obstruction for PE and PM Weyl curvature tensors]{A Pontryagin class obstruction for purely electric and purely magnetic Weyl curvature tensors}

\author[T.~de Kok]{Thijs de Kok\, \orcidlink{0009-0004-6433-3645}}
\address{Department of Mathematics, IMAPP, Radboud University, PO Box 9010, Postvak 59, 6500 GL Nijmegen, The Netherlands}
\email{\href{mailto:thijs.dekok@ru.nl}{thijs.dekok@ru.nl}}
\thanks{This version of the article has been accepted for open access publication in \emph{Letters in Mathematical Physics}. The Version of Record is available online at: \url{http://dx.doi.org/10.1007/s11005-026-02152-x}.}
\thanks{\emph{Acknowledgements:} This publication is part of the Vidi project with file number VI.Vidi.233.024 which is financed by the Dutch Research Council (NWO) under the grant https://doi.org/10.61686/LGCGZ85275. The author would also like to thank his supervisor, A. Burtscher, for the many discussions and the very careful proofreading of earlier versions of this work.}

\date{\today}

\numberwithin{equation}{section}

\newtheorem{innercustomgeneric}{\customgenericname}
\providecommand{\customgenericname}{}
\newcommand{\newcustomtheorem}[2]{%
	\newenvironment{#1}[1]
	{%
		\renewcommand\customgenericname{#2}%
		\renewcommand\theinnercustomgeneric{##1}%
		\innercustomgeneric
	}
	{\endinnercustomgeneric}
}

\theoremstyle{plain}
\newtheorem{thm}{Theorem}[section]
\newtheorem*{thm*}{Theorem}
\newtheorem{cor}[thm]{Corollary}
\newtheorem*{cor*}{Corollary}
\newtheorem{prop}[thm]{Proposition}
\newtheorem{lem}[thm]{Lemma}

\newcustomtheorem{customthm}{Theorem}

\theoremstyle{definition}
\newtheorem{defn}[thm]{Definition}
\newtheorem{ex}[thm]{Example}

\theoremstyle{remark}
\newtheorem{rem}[thm]{Remark}

\newcommand{\id}{\operatorname{id}}

\newcommand{\eps}{\varepsilon}
\newcommand{\II}{\operatorname{II}}

\newcommand{\tr}{\operatorname{tr}}
\newcommand{\sgn}{\operatorname{sgn}}
\newcommand{\End}{\operatorname{End}}

\newcommand{\fl}[1]{\lfloor #1 \rfloor}
\newcommand{\ol}[1]{\overline{#1}}

\DeclareMathOperator{\Ric}{Ric}

\newcommand\restr[2]{{% we make the whole thing an ordinary symbol
		\left.\kern-\nulldelimiterspace % automatically resize the bar with \right
		#1 % the function
		\vphantom{\big|} % pretend it's a little taller at normal size
		\right|_{#2} % this is the delimiter
}}

\def\R{\mathbb{R}}

\def\C{\mathbb{C}}
\def\Z{\mathbb{Z}}

\def\eps{\varepsilon}

\makeatletter
\newcommand*\owedge{\mathpalette\@owedge\relax}
\newcommand*\@owedge[1]{%
	\mathbin{%
		\ooalign{%
			$#1\m@th\bigcirc$\cr
			\hidewidth$#1\m@th\wedge$\hidewidth\cr
		}%
	}%
}
\makeatother

\newcounter{char}
\loop\ifnum\value{char}<27
\expandafter\edef\csname c\Alph{char}\endcsname{\noexpand\mathcal{\Alph{char}}}  % mathcal
\expandafter\edef\csname d\Alph{char}\endcsname{\noexpand\mathbb{\Alph{char}}}   % mathbb
\expandafter\edef\csname f\Alph{char}\endcsname{\noexpand\mathfrak{\Alph{char}}} % mathfrak
\expandafter\edef\csname e\Alph{char}\endcsname{\noexpand\mathsf{\Alph{char}}}   % mathsf
\expandafter\edef\csname s\Alph{char}\endcsname{\noexpand\mathscr{\Alph{char}}}  % mathscr
\addtocounter{char}{1}
\repeat

\begin{document}
	%%%%% TITLE AND ABSTRACT %%%%%
	
	\begin{abstract}
		Do all manifolds that admit Lorentzian metrics also admit such metrics that have a purely electric (PE) or purely magnetic (PM) Weyl curvature tensor?
		To (partially) answer this question, we show that for all algebraic curvature tensors on a $4k$-dimensional scalar product space that are even or odd under the action of a orientation-reversing isometry, the products of Pontryagin forms that land in the top-degree exterior power of the dual vector space vanish. We use this to derive the vanishing of all products of Pontryagin classes that land in the top-degree de Rham cohomology of a $4k$-dimensional pseudo-Riemannian manifold with a PE or PM Riemann or Weyl curvature tensor. For compact manifolds, this gives nontrivial cohomological obstructions to the existence of such pseudo-Riemannian metrics with globally PE or PM Riemann or Weyl curvature tensors. These obstructions can be linked to the existence of Lorentzian metrics of several Petrov subtypes, which play an important role in classifying exact solutions to the Einstein equations. Moreover, they can be applied to foliations by nondegenerate umbilic hypersurfaces, which may appear as timeslices of spacetimes.
	\end{abstract}
	
	\maketitle
	
	%\setcounter{tocdepth}{1}
	%\setcounter{secnumdepth}{1}
	%\tableofcontents
	%%%%%% INTRODUCTION %%%%%%
	\section{Introduction}
	\label{sect: intro}
	One of the big efforts in general relativity is finding and classifying exact solutions to the Einstein equations. By now, a vast amount of literature is available on this topic. See, e.g.,\ the book by Stephani et al \cite{Kramer03Exactsolutions}, that gives a survey on all exact solutions known up to 1999. A major tool for classifying exact solutions is by posing conditions on the structure of their curvature tensors. One approach is via the splitting of the Riemann/Weyl curvature tensor of a 4-dimensional Lorentzian manifold satisfying the vacuum Einstein equations into an ``electric'' and ``magnetic'' part, based on a fixed timelike unit vector field. This was introduced already in the 1950s by Matte \cite[\S 4]{Matte1953nouvellessolutions}. The terminology is motivated by the observation that the second Bianchi identity becomes a Maxwell-like equation with source terms \cite[\S 5]{Matte1953nouvellessolutions}. 
	
	On a 4-dimensional oriented Lorentzian manifold, this splitting is performed \cite[Ch.\ 3.5]{Kramer03Exactsolutions} by first introducing the complex-valued \emph{self-dual Weyl curvature tensor} $\widetilde{W}$ with components 
	\[
	\tensor{\widetilde{W}}{_a_b_c_d} = \tensor{W}{_a_b_c_d} + \frac{1}{2}i\tensor{\eps}{_a_b_e_f}\tensor{W}{^e^f_c_d},
	\]
	where $\eps$ is the totally anti-symmetric Levi-Civita symbol satisfying $\tensor{\eps}{_0_1_2_3} = -1$ with respect to a local oriented orthonormal basis $(e_0, e_1, e_2, e_3)$. Choosing a unit timelike vector field $T$ on $M$, we can then define the tensor fields $E$ and $B$ as 
	\[
	\tensor{E}{_a_c} + i\tensor{B}{_a_c} := \tensor{\widetilde{W}}{_a_b_c_d}\tensor{T}{^b}\tensor{T}{^d},
	\]
	which are the \emph{electric} and \emph{magnetic} parts of the Weyl curvature tensor, respectively. The electric part $E$ captures gravitational effects that resemble tidal forces from Newtonian mechanics, whereas the magnetic part $B$ has no Newtonian counterpart \cite[p.\ 607--608]{Ellis1971Relativisticcosmology}.
	
	The Weyl curvature tensor is \emph{purely electric (PE)} at $p \in M$ if its magnetic part $B$ vanishes and \emph{purely magnetic (PM)} if its electric part $E$ vanishes. Note that this is a pointwise property. PE spacetimes arise naturally in studying general relativity: all spacetimes with a unit timelike congruence that is shear-free and irrotational have PE Weyl curvature tensors with respect to the congruence \cite{Trumper1965typeIgravfields}. This includes static spacetimes \cite[Ch.\ 6.2]{Kramer03Exactsolutions}, warped products of the form $\R\times_f\Sigma$ (this follows from a variation of \cite[Prop.\ 42]{ONeill83SemiRiemannian} for the Weyl curvature tensor), and spacetimes with irrotational perfect fluids \cite{Barnes1973perfectfluids, Barnes1989irrfluidswithPEWeyl}. On the contrary, some examples of PM spacetimes are known \cite{Arianrhod1994magneticcurvatures, Lozanovski1999irrperffluidwithPMWeyl, Lozanovski1999perffluidsPMWeyl} but PM spacetimes are rather elusive in general and do not arise naturally. For example, Van den Bergh 	\cite{VdBergh2003PGMvacuumspacetimes} showed that there are no PM vacuum solutions to the Einstein equations and Lesame \cite{Lesame1995orrdustPM} showed that there are no irrotational dust solutions with PM Weyl curvature tensors. More generally, Maartens, Lesame and Ellis \cite{Maartens1998AntiNewtonianlikeuniverses} showed that metrics with PM Weyl curvature tensors are inconsistent with the Einstein equations in general. 
	
	A classical result is that a manifold $M$ admits a Lorentz metric if and only if it is either noncompact or is compact and has Euler-characteristic equal to 0. A natural follow-up question to ask is whether a manifold $M$ that admits Lorentzian (or in general pseudo-Riemannian) metrics also admits such metrics with PE or PM Weyl curvature tensors. For 4-dimensional Lorentzian manifolds, the property of having a PE or PM Weyl curvature tensor was found to be closely related to the Petrov type of the Weyl curvature tensor. In particular, the Weyl curvature tensor $W$ of a 4-dimensional Lorentzian manifold $M$ is PE/PM at $p \in M$ if it has Petrov type $O$, $I$ or $D$ with real/imaginary eigenvalues \cite[Rem.\ 3.8]{Hervik2013minimaltensors} \cite[p.\ 1556]{McIntosh1994PEPMWeyl} \cite{Trumper1965typeIgravfields, Wylleman2006Completeclassification}. Furthermore, Avez \cite{Avez70charclassesandWeyl} showed in the 1970s that geometric properties of the Weyl curvature tensor can be related to the smooth and topological structure of the manifold via the Pontryagin forms and classes. The \emph{Pontryagin forms} of a pseudo-Riemannian manifold $(M, g)$ are differential forms 
	\[
	\varpi_k(M) \in \Omega^{4k}(M),
	\]
	which are constructed from the Riemann curvature tensor of the Levi-Civita connection of $(M, g)$. The Pontryagin forms are closed differential forms and therefore define classes in the de Rham cohomology of the manifold $M$ \cite[p.\ 296]{Milnor74charclasses}. The resulting \emph{Pontryagin classes} 
	\[
	p_k(M):=[\varpi_k(M)] \in H^{4k}_{dR}(M)
	\]
	do not depend on the chosen metric $g$ and are proper invariants of $M$ \cite[p.\ 298]{Milnor74charclasses}. We refer to Section \ref{sect: Pontryagin forms on pseudo-Riemannian manifolds} for a precise introduction of the Pontryagin forms and classes. Clearly, the Pontryagin form $\varpi_k(M)$ and hence the Pontryagin class $p_k(M)$ vanish if $4k$ is greater than the dimension of $M$. In particular, for 4-dimensional manifolds $M$, only the first Pontryagin form and class may be nontrivial. Avez showed that if $M$ admits a locally conformally flat metric---that is, with vanishing Weyl curvature tensor $W$, which is by definition of Petrov type $O$---then the first Pontryagin form $\varpi_1(M)$ (and therefore also $p_1(M)$) vanishes. Later, Zund \cite{Zund1968topoaspectsBelPetrov} and Porter and Thompson \cite{Porter71topoinGR} extended the result by Avez to other Petrov types, including type $D$ with real or imaginary eigenvalues, which are PE or PM. This shows that on 4-dimensional Lorentzian manifolds, the Pontryagin class $p_1(M)$ may function as an obstruction to the existence of Lorentzian metrics with PE or PM type $D$ Weyl curvature tensors. Recently, Hervik, Ortaggio and Wylleman \cite{Hervik2013minimaltensors} observed that for the Weyl curvature tensor, being PE or PM with respect to a timelike unit vector $T$ is equivalent to being even or odd under the action of the orthogonal reflection along the hyperplane $T^{\bot}$. Using their characterization, the electric-magnetic splitting naturally extends to algebraic curvature tensors on vector spaces of arbitrary signature and dimension and can be performed for all vectors $T$ that are not null. 
	
	In this paper, we discuss in what way the Pontryagin classes of a manifold $M$ give an obstruction to the existence of PE or PM Weyl curvature tensors in the sense of Hervik et al. In Section \ref{sect: main result}, we will obtain the following precise restrictions on the Pontryagin forms and classes of a manifold admitting pseudo-Riemannian metrics with PE or PM Weyl curvature tensors. To the best of our knowledge, this is the first result relating the property of being PE or PM to the Pontryagin classes of the manifold directly. For the precise statement of the presented results, we refer to the corresponding results in the text.
	
	\begin{thm}[Theorem \ref{thm: vanishing Pontryagin forms mfds}]
		\label{thm: main result}
		Let $(M, g)$ be a $4k$-dimensional pseudo-Riemannian manifold and suppose that the Weyl curvature tensor $W$ is PE or PM at each point. Then for all multi-indices $\alpha = (\alpha_1, \ldots, \alpha_l) \in \Z_{> 0}^l$ with $\alpha_1+\cdots+\alpha_l = k$, the $4k$-forms
		\begin{equation*}
			\varpi_{\alpha}(W) := \varpi_{\alpha_1}(W)\wedge\cdots\wedge\varpi_{\alpha_l}(W) = 0 \in \Omega^{4k}(M).
		\end{equation*}
		In particular, the class
		\begin{equation*}
			p_{\alpha}(M):= p_{\alpha_1}(M)\smile\cdots\smile p_{\alpha_l}(M) = 0 \in H^{4k}_{dR}(M).
		\end{equation*}
	\end{thm}

	Our proof of Theorem \ref{thm: main result} is purely algebraic in the sense that it only depends on the algebraic symmetries of the Weyl curvature tensor and the algebraic restrictions that being PE or PM yield for the Weyl curvature tensor. Therefore, we prove Theorem \ref{thm: main result} generally for all algebraic curvature tensors that satisfy a PE- or PM-like algebraic condition. We refer to Section \ref{sect: main result} for the general set-up. The main idea of the proof is to show that if an algebraic curvature tensor $C$ on a $4k$-dimensional scalar product space $(V, g)$ is even or odd under the action of a linear isometry of $g$, then the maps induced by the isometry on the exterior powers of the dual space $V^*$ fix the Pontryagin forms of the algebraic curvature tensor. Consequently, also products of such Pontryagin forms are fixed by these maps. The result then follows from the observation that if the determinant of such isometry is $-1$, then the only element in $\wedge^{4k}V^*$ that is fixed is 0.
	
	Using that the Weyl curvature tensor of a 4-dimensional Lorentzian manifold of Petrov type $I$ with real or purely imaginary eigenvalues is PE or PM \cite[p.\ 1556]{McIntosh1994PEPMWeyl}, we can extend the results of Avez, Zund, Porter and Thompson and also add Petrov type $I$ with real or imaginary eigenvalues to the list of Petrov (sub)types that guarantee the vanishing of the first Pontryagin class $p_1(M)$ of the manifold $M$.
	
	\begin{thm}[Theorem \ref{thm: vanishing of Pontryagin form for Petrov types}]
		\label{thm: result on Petrov types}
		Let $W$ be the Weyl curvature tensor of a 4-dimensional Lorentzian manifold $(M, g)$. If at all points $p\in M$ the Weyl curvature tensor $W$ has Petrov type $I$ with real or purely imaginary eigenvalues, then $p_1(M) = 0$.
	\end{thm}
	
	Finally, we discuss an application of Theorem \ref{thm: main result} to the existence of foliations by nondegenerate umbilic hypersurfaces. Recall that a nondegenerate hypersurface $\Sigma$ of a pseudo-Riemannian manifold $(M, g)$ is \emph{umbilic} if the scalar second fundamental form $h$ of the embedding $\Sigma \hookrightarrow M$ is of the form $h = f\sigma$, where $f\in C^{\infty}(\Sigma)$ and $\sigma$ is the induced metric on $\Sigma$. Umbilic hypersurfaces are studied both in the setting of Riemannian geometry, see e.g.\ \cite{Cairns1990Totallyumbilicfoliations, Gomes2008Umbilicalfoliations, Langevin2008Conformalgeometryfoliations} for just a few examples, and in the setting of Lorentzian geometry, where umbilic \emph{spacelike} hypersurfaces may appear as time slices of spacetimes \cite{Avalos2025energypositivityspacetimesexpanding, Brasil2025spacelikefoliations,Ferrando1992inhomogeneousspaces}.
	At such a nondegenerate umbilic hypersurface, the Weyl curvature tensor is even with respect to the orthogonal reflection along $T\Sigma$ \cite[\S 2]{Sharma2023conformalflatness}. This allows us to apply Theorem \ref{thm: main result}.
	
	\begin{thm}[Theorem \ref{thm: products of Pontryagin classes for umbilic foliations}]
		\label{thm: result on umbilic hypersurfaces}
		Let $(M, g)$ be $4k$-dimensional pseudo-Riemannian manifold. Let $\Sigma \hookrightarrow M$ be a nondegenerate umbilic hypersurface. Then for all multi-indices $\alpha = (\alpha_1, \ldots, \alpha_l)  \in \Z_{> 0}^l$ with $\alpha_1+\cdots+\alpha_l = k$, the $4k$-forms $\varpi_{\alpha}(W)$ vanish at $\Sigma \subset M$.
		
		In particular, if $(M, g)$ is foliated by such hypersurfaces, then for all such multi-indices $\alpha$, the $4k$-form $\varpi_{\alpha}(W)$ vanishes and therefore $p_{\alpha}(M) = 0$.
	\end{thm}
	
	When formulated contrapositively, Theorems \ref{thm: main result}, \ref{thm: result on Petrov types} and \ref{thm: result on umbilic hypersurfaces} pose obstructions to the existence of pseudo-Riemannian metrics with globally PE or PM Weyl curvature tensors, Lorentzian metrics with these Petrov subtypes and pseudo-Riemannian metrics for which $M$ is foliated by nondegenerate umbilic hypersurfaces, respectively.
	
	\subsection*{Outline}
	The paper is structured as follows. In Section \ref{sect: preliminaries}, we briefly recall the notion of an algebraic curvature tensor on a vector space, how it defines a curvature operator on bivectors and we introduce Thorpe's higher curvature operators acting on the higher exterior powers of the vector space. We also review how the decomposition of the Riemann curvature tensor into its Weyl, Ricci and scalar components carries over to algebraic curvature tensors. We will conclude Section \ref{sect: preliminaries} with recalling the construction of the Pontryagin forms and classes of a manifold, which motivates defining the Pontryagin forms of an algebraic curvature tensor. We present two important properties of the Pontryagin forms: the first being that we can express the Pontryagin forms succinctly using the higher curvature operators (Theorem \ref{thm: expression Pontryagin form}) and the second being that the Pontryagin forms of an algebraic curvature tensor $C$ are equal to the Pontryagin forms of its Weyl curvature tensor $W_C$ (Theorem \ref{thm: Pontryagin forms equal for weyl curvature}). The former is mentioned a few times in the literature, but without proof or specifically for the Riemannian setting. Therefore, we provide a proof of the statement in Appendix \ref{app: proof expression Pontryagin forms}. In Section \ref{sect: PE and PM general}, we will first discuss the approach to PE and PM curvature tensors by Hervik et al., characterizing PE and PM curvature tensors as the curvature tensors that are even or odd under the action of a reflection in the hyperplane orthogonal to a timelike unit vector. In Section \ref{sect: main result}, we prove our main result Theorem \ref{thm: main result} and give an example of a manifold that admits Lorentzian metrics but none with a PE or PM Weyl curvature tensor. We conclude the paper by discussing resulting obstructions for the existence of some Petrov subtypes (Theorem \ref{thm: result on Petrov types}) and foliations by nondegenerate umbilic hypersurfaces (Theorem \ref{thm: result on umbilic hypersurfaces}).
	
	%%%%%%%% SECTIONS %%%%%%%%
	\section{Preliminaries}
	\label{sect: preliminaries}
	In this section, we first recall the notion of an algebraic curvature tensor on a vector space and the associated (higher) curvature operators acting on the even exterior powers of that vector space. In Section \ref{sect: decomposition of algebraic curvature tensors}, we review the familiar splitting of the Riemann curvature tensor into its Weyl, Ricci and scalar components in the setting of algebraic curvature tensors. Finally, in Section \ref{sect: Pontryagin forms on pseudo-Riemannian manifolds}, we conclude by introducing the Pontryagin forms and classes of a manifold.
	
	Unless specified otherwise, we follow the conventions of Lee \cite{Lee18Riemannmfds}. Our manifolds are second countable, Hausdorff and $C^{\infty}$. Let $M$ be an $n$-dimensional manifold and let $g$ be a pseudo-Riemannian metric on $M$ with signature $(p, q)$, where $p$ denotes the number of timelike directions.
	
	\subsection{Algebraic curvature tensors}
	\label{sect: algebraic curvature tensors}
	Let $\nabla$ be the Levi-Civita connection of $g$. We use the convention of Lee \cite{Lee18Riemannmfds} and define the \emph{Riemann (curvature) tensor} for all $U, V, X \in \fX(M)$ by
	\begin{equation}
		\label{eq: (1, 3)-curvature}
		R(U, V)X = [\nabla_U, \nabla_V]X-\nabla_{[U, V]}X,
	\end{equation}
	which is $C^{\infty}(M)$-linear in $U$, $V$ and $X$ and therefore defines a $(1,3)$-tensor field on $M$. Alternatively, we can use the metric $g$ to express $R$ as a $(0,4)$-tensor field $Rm$ on $M$ defined for all $U, V, X, Y \in \fX(M)$ by
	\begin{equation}
		\label{eq: (0, 4)-curvature}
		Rm(U, V, X, Y) = g(R(U, V)X, Y),
	\end{equation}
	which we also refer to as the Riemann (curvature) tensor. Based on the symmetries of the Riemann tensor, we can introduce the algebraic curvature tensors as the $(0, 4)$-tensors that satisfy the same algebraic symmetries.
	
	\begin{defn}
		\label{algebraic curvature tensor (fields)}
		 Let $V$ to be a finite-dimensional vector space over $\R$. An \emph{algebraic curvature tensor} on $V$ is a $(0,4)$-tensor $C$ on $V$ that for all $u, v, x, y \in V$ satisfies the symmetries:
		\begin{enumerate}
			\item $C(u, v, x, y) = -C(v, u, x, y) = -C(u, v, y, x)$,
			\item $C(u, v, x, y) = C(x, y, u, v)$ and
			\item $C(u, v, x, y) + C(x, u, v, y) + C(v, x, u, y) = 0$.
		\end{enumerate}
		We denote the vector space of algebraic curvature tensors on $V$ by $\cC(V)$. An \emph{algebraic curvature tensor field} on a manifold $M$ is a $(0,4)$-tensor field $C$ on $M$ such that $C_p$ is an algebraic curvature tensor on $T_pM$ for all $p\in M$.
	\end{defn}
	
	We recall the associated curvature operator of an algebraic curvature tensor and its higher analogues. Let $(V, g)$ be a scalar product space and let $C \in \cC(V)$ be an algebraic curvature tensor. Recall that the scalar product $g$ on $V$ induces a nondegenerate scalar product $\langle\cdot,\cdot\rangle_g$ on $\wedge^kV$ that is uniquely defined by 
	\begin{equation}
		\label{eq: induced inner product on exterior algebra}
		\langle v_1\wedge\cdots\wedge v_k, w_1\wedge\cdots\wedge w_k\rangle_g = \det(\{g(v_i, w_j)\}_{1\leq i, j\leq k})
	\end{equation}
	for all $v_i, w_j \in V$ and extended bilinearly to all of $\wedge^k V$ \cite[Ch.\ 4.8, 5.4]{Greub1978multilinalg}.
	
	\begin{defn}
		\label{def: curvature operator}
		The \emph{curvature operator} associated with the algebraic curvature tensor $C$ is the unique linear operator $\hat{C}$ on $\wedge^2V$ that is defined by
		\[
		\langle \hat{C}(u\wedge v), x\wedge y\rangle_g = C(u, v, x, y),
		\]
		for all $u, v, x, y \in V$.
	\end{defn}
	
	\begin{rem}
		\label{rem: different sign convention}
		In Definition \ref{def: curvature operator}, we use the opposite sign as in Lee \cite[p.\ 262]{Lee18Riemannmfds}. This is chosen so to stay more consistent with sign conventions of other references in later parts of this paper.
	\end{rem}
	
	The curvature operator has higher analogues, which are linear operators on the vector spaces $\wedge^{2k}V$ and were introduced by Thorpe in \cite{Thorpe1964sectionalcurvatures}. These are constructed by extending the curvature operator $\hat{C}$ to the vector spaces $\wedge^{2k}V$ by taking powers under the product of the mixed exterior algebra \cite[Ch.\ 6]{Greub1978multilinalg}. For simplicity, we do not wish to introduce all notation regarding the mixed exterior algebra and follow the construction of the higher curvature operators by Bivens \cite{Bivens1982curvatureandcharclasses}.
	
	\begin{defn}
		\label{def: multiplication of exterior endomorphisms}
		Let $A \in \End(\wedge^k V)$ and $B \in \End(\wedge^l V)$. We define $A*B\in \End(\wedge^{k+l} V)$ as the unique linear operator satisfying
		\[
		(A*B)(x_1\wedge\cdots\wedge x_{k+l}) = \frac{1}{k!l!}\sum_{\sigma \in S_{k+l}}\sgn(\sigma)A(x_{\sigma(1)}\wedge\cdots\wedge x_{\sigma(k)})\wedge B(x_{\sigma(k+1)}\wedge\cdots\wedge x_{\sigma(k+l)}),
		\]
		for all $x_1, \ldots, x_{k+l} \in V$. Here, $S_{k+l}$ denotes the permutation group on $k+l$ elements and $\sgn(\sigma)$ denotes the sign of the permutation $\sigma$.
	\end{defn}
	
	For simplicity, let $[n]=\{1, \ldots, n\}$. If $I =(i_1, \ldots, i_k)\in [n]^k$ is a \emph{multi-index} of elements in $[n]$ of \emph{length} $k$ and $e = (e_1, \ldots, e_n)$ is a basis for the vector space $V$, then we denote by $e_I = e_{i_1}\wedge\cdots\wedge e_{i_k}$ the iterated wedge-product. Recall that $e_I$ vanishes if $I$ has a repeated index. Note that the $k$-vectors $e_I$ with increasing multi-index $I$ of length $k$ form a basis of $\wedge^kV$ and therefore the collection of all $k$-vectors of the form $e_I$ form a spanning set of $\wedge^kV$. Expanding a $k$-vector over all such $e_I$ has the advantage that we do not need to control the ordering of the multi-indices $I$, at the cost of working with a non-unique expansion.
	
	Note that the following result is immediate after identifying $*$ as the multiplication on the diagonal subalgebra of the mixed exterior algebra of $V$ \cite[Ch.\ 6.2]{Greub1978multilinalg}.
	
	\begin{lem}
		\label{lem: properties *}
		Let $V$ be an $n$-dimensional real vector space and let $0\leq k, l \leq n$ be arbitrary. Then the operation 
		\[
		*\colon \End(\wedge^kV)\times \End(\wedge^lV)\rightarrow \End(\wedge^{k+l}V)
		\]
		is associative, commutative and $\R$-bilinear.
		
		If $e = (e_1, \ldots, e_n)$ is a basis of $V$ and $A \in \End(\wedge^k V)$ and $B \in \End(\wedge^l V)$ are expanded as
		\[
		A = \sum_{I \in [n]^k}\alpha^I\otimes e_I \qquad\text{and}\qquad B = \sum_{J\in [n]^l}\beta^J\otimes e_J
		\]
		with $\alpha^I \in \wedge^kV^*$ and $\beta^J \in \wedge^lV^*$, then 
		\begin{equation}
			\label{eq: expansion * product}
			A*B = \sum_{\substack{I \in [n]^k\\ J\in [n]^l}}(\alpha^I\wedge \beta^J)\otimes (e_I\wedge e_J).
		\end{equation}
		
		Moreover, for every $\theta \in \End(V)$:
		\begin{align}
		\label{eq: compatibility * back}
			(A*B)\circ \wedge^{k+l}\theta &= (A\circ \wedge^k\theta)*(B\circ \wedge^l\theta), \\
		\label{eq: compatibility * front}
			\wedge^{k+l}\theta \circ (A*B)  &= (\wedge^k\theta \circ A)*(\wedge^l\theta \circ B).
		\end{align}
	\end{lem}
	\begin{proof}
		$\R$-bilinearity of $*$ is clear from $\R$-bilinearity of the wedge product. Equation \eqref{eq: expansion * product} follows directly by writing out $A*B$ and using the expansion of $A$ and $B$. Associativity and commutativity of $*$ are then easily deduced from \eqref{eq: expansion * product}. Denote by $\theta^* \in \End(V^*)$ the induced dual map. Then using \eqref{eq: expansion * product}, we see that
		\begin{align*}
			(A*B)\circ \wedge^{k+l}\theta &= \sum_{\substack{I \in [n]^k\\ J\in [n]^l}}\wedge^{k+l}\theta^*(\alpha^I\wedge \beta^J)\otimes (e_I\wedge e_J) \\
			&=\sum_{\substack{I \in [n]^k\\ J\in [n]^l}}(\wedge^{k}\theta^*(\alpha^I)\wedge(\wedge^{l}\theta^*(\beta^J)))\otimes (e_I\wedge e_J) = (A\circ \wedge^k\theta)*(B\circ \wedge^l\theta),
		\end{align*}
		which proves \eqref{eq: compatibility * back}. Equation \eqref{eq: compatibility * front} follows from a similar computation.
	\end{proof}
		
	\begin{defn}
		\label{def: higher curvature operator}
		Let $C$ be an algebraic curvature tensor on $V$ and denote its curvature operator by $\hat{C}$. Then we define the \emph{$k$th curvature operator} associated to $C$ as
		\[
		\hat{C}_{2k} = \hat{C}^{*k} \in \End(\wedge^{2k}V).
		\]
	\end{defn}
	
	\begin{rem}
		The $k$th curvature operator is labeled with the index $2k$ to keep consistent with conventions in existing literature on higher curvature operators (e.g.\ \cite{Bivens1982curvatureandcharclasses, Thorpe1969RemarksGaussBonnet}).
	\end{rem}
	
	\begin{rem}[Curvature operators on manifolds] 
	\label{rem: curvature operators on manifolds}
	If $C$ is an algebraic curvature tensor field on a manifold $M$, applying Definition \ref{def: curvature operator} pointwise on each $T_pM$ yields a vector bundle endomorphism
	\[
	\hat{C}\colon \wedge^2TM\rightarrow \wedge^2TM.
	\]
	Smoothness of $\hat{C}$ can easily be verified by considering its action on bivector fields, see e.g. \cite[p.\ 262]{Lee18Riemannmfds} (note the difference in sign convention, see Remark \ref{rem: different sign convention}).
	Similarly, applying Definition \ref{def: multiplication of exterior endomorphisms} on each $T_pM$ yields a $C^{\infty}(M)$-bilinear operation 
	\[
	*\colon \End(\wedge^kTM)\times \End(\wedge^lTM)\rightarrow \End(\wedge^{k+l}TM).
	\]
	Therefore, we obtain vector bundle endomorphisms
	\[
	\hat{C}_{2k} = \hat{C}^{*k}\colon \wedge^{2k}TM\rightarrow \wedge^{2k}TM
	\]
	as higher curvature operators.
	\end{rem}
	
	\subsection{Decomposition of algebraic curvature tensors}
	\label{sect: decomposition of algebraic curvature tensors}
	The Riemann curvature tensor $Rm$ of a pseudo-Riemannian manifold $(M, g)$ can be decomposed into simpler curvature objects, namely the Ricci curvature, scalar curvature and Weyl curvature tensor of $(M, g)$ \cite[Ch.\ 7]{Lee18Riemannmfds}. In general, we can apply this construction to decompose any algebraic curvature tensor into its Weyl, Ricci and scalar curvature using a choice of scalar product $g$. 
	
	Again, let $(V, g)$ be an $n$-dimensional scalar product space.	First, recall that $g$ gives rise to traces $\tr^{i, j}_g\colon T^{(0, k+2)}(V)\rightarrow T^{(0, k)}(V)$ by contracting the $i$th and $j$th entry. For example, if $C \in \cC(V)$ is an algebraic curvature tensor and therefore a $(0, 4)$-tensor on $V$, the contraction of its first and fourth entry is given by  
	\begin{equation*}
		\tensor{\tr_g^{1, 4}(C)}{_k_l} = \sum_{i, j}\tensor{g}{^i^j}\tensor{C}{_i_k_l_j},
	\end{equation*}
	where $\tensor{g}{^i^j}$ is the inverse metric, satisfying $\tensor{g}{^i^j}\tensor{g}{_j_k} = \tensor{\delta}{^i_k}$. Denote the symmetric bilinear forms on $V$ by $\Sigma^2(V^*)$ and the trace-free symmetric bilinear forms by $\Sigma_0^2(V^*)$. If $h, k \in \Sigma^2(V^*)$ are two symmetric bilinear forms on $V$, then their \emph{Kulkarni--Nomizu product} is the $(0,4)$-tensor
	\[
	(h\owedge k) (u,v,x,y) = h(u, y)k(v, x) + h(v,x)k(u,y) - h(u,x)k(v, y) - h(v, y)k(u, x),
	\]
	which is an algebraic curvature tensor.
	
	\begin{defn}
		\label{def: curvature components}
		Let $C$ be an algebraic curvature tensor. Then we define its \emph{Ricci curvature} as $\Ric_C = \tr_g^{1,4}(C) \in \Sigma^2(V^*)$ and its \emph{scalar curvature} as $S_C = \tr_g(\Ric_C) \in \R$. We also define its \emph{Schouten tensor} as
		\[
		P_C = \frac{1}{n-2}\left(\Ric_C-\frac{S_C}{2(n-1)}g\right) \in \Sigma^2(V^*),
		\] and its \emph{Weyl curvature tensor} as 
		\[
		W_C = C - P_C\owedge g \in \cC(V).
		\]
	\end{defn}	

	An easy computation shows that the Weyl curvature tensor of any algebraic curvature tensor $C$ is trace-free in the sense that $\tr^{1, 4}_g(W_C) = 0$. Moreover, if $C$ is trace-free already, then $C = W_C$. Therefore, the following definition makes sense.
	
	\begin{defn}
		\label{def: Weyl curvature tensors}
		A \emph{Weyl curvature tensor} $W$ is an algebraic curvature tensor that satisfies $\tr^{1, 4}_g(W) = 0$. We denote the vector space of Weyl curvature tensors by $\cW(V, g) = \cC(V)\cap \ker(\tr^{1, 4}_g)$.
	\end{defn}
	
	\begin{rem}[Curvature objects on manifolds]
		\label{rem: components for tensor fields}
		If instead we have an algebraic curvature tensor field $C$ on a manifold $(M, g)$, then by applying Definition \ref{def: curvature components} pointwise, we obtain smooth symmetric $(0, 2)$-tensors fields $\Ric_C$ and $P_C$, a smooth function $S_C$ and an algebraic curvature tensor field $W_C$, which carry the same names as the corresponding objects in Definition \ref{def: curvature components}.
	\end{rem}
	
	We recall the decomposition of an algebraic curvature tensor into its Schouten and Weyl curvature tensor.
	
	\begin{prop}[{\cite[Prop.\ 7.24]{Lee18Riemannmfds}}]
		\label{prop: curvature decomposition}
		For every algebraic curvature tensor $C$ on a scalar product space $(V, g)$ of dimension $n\geq 3$, the Weyl curvature tensor $W_C$ is a Weyl curvature tensor in the sense of Definition \ref{def: Weyl curvature tensors}, and $C = W_C + P_C\owedge g$ is the orthogonal decomposition of $C$ corresponding to $\cC(V) = \cW(V, g)\oplus \cW(V, g)^{\bot}$.
	\end{prop}
	
	\begin{rem}
		\label{rem: refined curvature decomposition theorem}
		Let
		\[
		O(V, g) := \{T \in \End(V): T^*g = g\}
		\]
		be the Lie group of linear isometries of $g$, which acts naturally on curvature tensors. This gives $\cC(V)$ the structure of an $O(V, g)$-module. The curvature decomposition in Proposition \ref{prop: curvature decomposition} can be refined to take into account the action of the group $O(V, g)$. It follows that the decomposition in Proposition \ref{prop: curvature decomposition} can be extended to
		\[
		\cC(V) = \cW(V, g)\oplus \cW(V, g)^{\bot} = \cW(V, g)\oplus \Sigma^2_0(V^*) \oplus \R,
		\]
		where the latter is an orthogonal decomposition into irreducible $O(V, g)$-modules. Note that this implies that the projections onto the submodules are $O(V, g)$-equivariant. 
		
		In this decomposition, the $\Sigma^2_0(V^*)$-summand represents the trace-free Ricci curvature $\accentset{\circ}{\Ric}_C=\Ric_C - \frac{S_C}{n}g$ of $C$ and the $\R$-summand represents the scalar curvature $S_C$ of $C$. For a proof, see \cite{Sing69curvatureofEinsteinspaces}; for more context, see \cite[Thm.\ 3.1]{BV09geometricrealizations}.
	\end{rem}
	
	\subsection{Pontryagin forms and higher curvature operators on pseudo-Riemannian manifolds}
	\label{sect: Pontryagin forms on pseudo-Riemannian manifolds}
	Pontryagin classes are a set of characteristic classes associated with real vector bundles. They are elements of the (de Rham) cohomology of the base space of the vector bundle and are important tools for distinguishing vector bundles both in algebraic topology and differential geometry. From a differential geometric perspective, characteristic classes have the interesting property that they can be computed using the curvature of an arbitrary connection on the vector bundle, but do not depend on the choice of connection. Consequently, the vanishing of some characteristic classes is sometimes a necessary (but certainly not always a sufficient) obstruction for the existence of other geometric structures on that vector bundle. For more context on characteristic classes, we refer to \cite{Bott1982diffformsinalgtop, Milnor74charclasses, Zhang2001ChernWeil}.
	
	In our main theorem, Theorem \ref{thm: vanishing Pontryagin forms mfds}, we will show that if the Weyl curvature tensor of a pseudo-Riemannian manifold has certain symmetries, then we can ensure the vanishing of specific products of Pontryagin classes. To prove the vanishing of these cohomology classes, we consider a specific differential form representing these cohomology classes and show that these vanish in Theorem \ref{thm: vanishing Pontryagin forms vector spaces}. To make the proof of Theorem \ref{thm: vanishing Pontryagin forms vector spaces} as clear as possible, the choice of representing differential form is important. Writing the representing differential forms in terms of the higher curvature operators is particularly fruitful for our purposes. Therefore, in this section we will first recall the construction of the Pontryagin classes of a vector bundle in terms of the curvature of a connection and then show how these can be represented using higher curvature operators.
	
	Let $M$ be a manifold. Let $E\rightarrow M$ be a real vector bundle over $M$ of rank $r$ and let $\nabla^E$ be a connection on $E$. Let $R^E$ denote the curvature tensor of $\nabla^E$, which is defined similar to \eqref{eq: (1, 3)-curvature}. Choose any local frame $e = (e_1, \ldots, e_r)$ of $E$ on some open subset $U\subseteq M$. The \emph{curvature matrix} $\Omega_U$ of $R^E$ on $U$ is the matrix of local 2-forms $\tensor{\Omega}{_i^j}\in \Omega^2(U)$ defined by 
	\begin{equation}
		\label{eq: 2-forms curvature matrix}
		R^E(X, Y)e_i = \sum_{j\in [r]} \tensor{\Omega}{_i^j}(X, Y)e_j.
	\end{equation}
	Clearly, if we choose a different frame $e'$ on $U$ such that $e_i = \sum_j\tensor{\tau}{_i^j}e'_j$ for some transition matrix $\tau \in GL_r(C^{\infty}(U))$, then standard linear algebra shows that the curvature matrix of $R^E$ on $U$ changes by 
	\[
	\Omega'_U = \tau^{-1}\Omega_U \tau.
	\]
	The Pontryagin classes are defined as the cohomology classes represented by differential forms on $M$ that are locally given as a polynomial combination of the local 2-forms $\tensor{\Omega}{_i^j}\in \Omega^2(U)$. To construct a globally well-defined differential form from these local curvature matrices, that does not depend on the choice of local frame, this polynomial must therefore be invariant under conjugation of the argument. Note that the determinant is such a polynomial. Indeed, the determinant, when viewed as map from $M_r(\R)\rightarrow \R$ is conjugation invariant and can be expressed as a polynomial in the matrix entries. Let $t\in \R$ and let $\sigma^{(r)}_1, \ldots, \sigma^{(r)}_r\colon M_r(\R)\rightarrow \R$ be the polynomials defined by 
	\[
	\det(I + tX) = 1 + \sum_{k \in [r]}t^k\sigma^{(r)}_k(X),
	\]
	for all $X \in M_r(\R)$. The polynomials $\sigma^{(r)}_k$ are homogeneous of degree $k$ and are invariant under conjugation of $X$ with elements of $GL_r(\R)$, since the determinant is. 
	
	We can now construct the Pontryagin classes as follows. If $\Omega_U \in M_r(\Omega^2(U))$ is the curvature matrix with respect to an arbitrary local frame $e$ on $U$, then the differential form $\sigma^{(r)}_k(\Omega_U) \in \Omega^{2k}(U)$ is well-defined since wedge products of differential 2-forms commute, and does not depend on the choice of $e$ by conjugation-invariance of $\sigma^{(r)}_k$. In particular, if $U$ and $V$ are two trivializing neighbourhoods for $E$, then $\sigma^{(r)}_k(\Omega_U)$ and $\sigma^{(r)}_k(\Omega_V)$ agree on $U\cap V$ and thus the differential forms $\sigma^{(r)}_k(\Omega_U)$ glue to a global differential form, denoted by $\sigma^{(r)}_k(\Omega_{\nabla^E}) \in \Omega^{2k}(M)$. It can be shown, that the differential form $\sigma^{(r)}_k(\Omega_{\nabla^E})$ is a de Rham cocycle \cite[p.\ 296]{Milnor74charclasses} and that if $\ol{\nabla}^E$ is a second connection on $E$, which induces the differential form $\sigma^{(r)}_k(\Omega_{\ol{\nabla}^E})$, then the differential form $\sigma^{(r)}_k(\Omega_{\nabla^E}) - \sigma^{(r)}_k(\Omega_{\ol{\nabla}^E})$ is a de Rham coboundary \cite[p.\ 298]{Milnor74charclasses}. This ensures that the following definition does not depend on the choice of connection.
	
	\begin{defn}
		\label{def: Pontryagin class E and M}
		Let $M$ be an $n$-dimensional manifold and let $E\rightarrow M$ be a rank $r$ vector bundle over $M$. Then for all $1\leq k\leq \fl{\frac{r}{2}}$, the \emph{$k$th Pontryagin class of $E$}, denoted by $p_k(E)$, is the class
		\[
		p_k(E) = \frac{1}{(2\pi)^{2k}}[\sigma^{(r)}_{2k}(\Omega_{\nabla^E})] \in H^{4k}_{dR}(M),
		\]
		where $\nabla^E$ is any connection on $E$ and $[-]$ denotes the corresponding de Rham cohomology class. We also define the $k$th Pontryagin class of $M$ as
		\[
		p_k(M) := p_k(TM).
		\]
	\end{defn}
	
	For the rest of this section, let $(V, g)$ be an $n$-dimensional scalar product space and let $C \in \cC(V)$ be an algebraic curvature tensor. Motivated by Definition \ref{def: Pontryagin class E and M}, we consider the following.
	
	\begin{defn}
		\label{def: Pontrayagin form}
		For all $1\leq k\leq \fl{\frac{n}{2}}$, we define the $k$th \emph{Pontryagin form} of $C$ by 
		\[
		\varpi_k(C) = \frac{1}{(2\pi)^{2k}}\sigma^{(n)}_{2k}(\Omega_C) \in \wedge^{4k}V^*,
		\]
		where $\Omega_C$ is the curvature matrix of $C$, when viewed as a $(1, 3)$-tensor, with respect to an arbitrary basis of $V$.
	\end{defn}
	
	\begin{rem}
		Note that the $\varpi_k(C)$ not only depends on $C$, but also on the scalar product $g$ as it is used to convert $C$ to a $(1, 3)$-tensor.
	\end{rem}
	
	As mentioned in the beginning of this section, we will show in Theorem \ref{thm: vanishing Pontryagin forms vector spaces} that certain symmetries of the curvature tensor $C$ ensure that products of specific Pontryagin forms of $C$ vanish. For the presentation of the argument, it is useful to rewrite the $k$th Pontryagin forms of $C$ in terms of the $k$th curvature operator of $C$. This was first done in the Riemannian setting by Stehney \cite[Thm.\ 4.1]{Stehney1973courbure}, based on work by Chern \cite[Thm.\ 2]{Chern1955charclassesRiemannian}. The corresponding result in the pseudo-Riemannian case is mentioned by Greub \cite[\S 4]{Greub1981PontryaginclassesandWeyl} without proof. Therefore, we provide this proof in Appendix \ref{app: proof expression Pontryagin forms}. 
	
	Note that the $k$th curvature operator is an endomorphism on $\wedge^{2k}V$, whereas the $k$th Pontryagin form is a $4k$-form on $V$. Therefore, consider the following $2k$-form constructed from two endomorphisms on $\wedge^k V$.	
	
	\begin{defn}
		\label{def: differential forms induced by exterior endomorphisms}
		Let $A, B \in \End(\wedge^k V)$. Then $F_k(A, B) \in \wedge^{2k}V^*$ is the $2k$-form defined by 
		\begin{multline}
		\label{eq: differential forms induced by exterior endomorphisms}
			F_k(A, B)(x_1\wedge\cdots\wedge x_{2k}) \\
			= \frac{1}{k!^2}\sum_{\sigma \in S_{2k}}\sgn(\sigma)\langle A(x_{\sigma(1)}\wedge\cdots\wedge x_{\sigma(k)}), B(x_{\sigma(k+1)}\wedge\cdots\wedge x_{\sigma(2k)})\rangle_g
		\end{multline}
		for all $x_1, \ldots, x_{2k} \in V$.
	\end{defn}
	
	Using $F_{2k}$, we can express the $k$th Pontryagin form of $C$ in terms of its $k$th curvature operator \cite{Greub1981PontryaginclassesandWeyl, Stehney1973courbure}.
	
	\begin{thm}
		\label{thm: expression Pontryagin form}
		For all $1\leq k\leq \fl{\frac{n}{2}}$,
		\begin{equation}
			\label{eq: Pontrayagin form in kth curvature operator}
			\varpi_k(C)	= \frac{1}{(2\pi)^{2k}k!^2}F_{2k}(\hat{C}^{*k}, \hat{C}^{*k}).
		\end{equation}
	\end{thm}
	
	A proof of Theorem \ref{thm: expression Pontryagin form} is given in Appendix \ref{app: proof expression Pontryagin forms}. We conclude this section by recalling the following result, which shows that the Pontryagin forms of an algebraic curvature tensor $C \in \cC(V)$ are equal to the Pontryagin forms of its Weyl curvature tensor $W_C$. This was first shown by Avez \cite{Avez70charclassesandWeyl} in 1970, based on a paper by Chern and Simons \cite{Chern1971cohomclasses}, and later by Greub \cite{Greub1981PontryaginclassesandWeyl} and also by Bivens \cite[Lem.\ 2.1.a]{Bivens1982curvatureandcharclasses}.
	
	\begin{thm}
		\label{thm: Pontryagin forms equal for weyl curvature}
		Let $C \in \cC(V)$ be an algebraic curvature tensor and let $W_C$ be the Weyl curvature tensor of $C$ as in Definition \ref{def: curvature components}. Then for all $1\leq k\leq \fl{\frac{n}{2}}$,
		\begin{equation}
			\label{eq: Pontrayagin form equal to Pontryagin of Weyl}
			\varpi_k(C) = \varpi_k(W_C).
		\end{equation}
	\end{thm}

	On manifolds, one can show that Theorem \ref{thm: Pontryagin forms equal for weyl curvature} implies the following result from Stehney \cite[Thm.\ 4.1]{Stehney1973courbure} and Greub \cite[\S 4]{Greub1981PontryaginclassesandWeyl}.
	
	\begin{cor}
		\label{cor: Pontryagin class manifold}
		Let $(M, g)$ be an $n$-dimensional pseudo-Riemannian manifold and let $Rm$ be its Riemann curvature tensor and $W$ its Weyl curvature tensor. Then for $1\leq k\leq \fl{\frac{n}{2}}$, the $k$th Pontryagin class of $M$ is given by 
		\[
		p_k(M) = \frac{1}{(2\pi)^{2k}k!^2}[F_{2k}(\hat{Rm}^{*k}, \hat{Rm}^{*k})] = \frac{1}{(2\pi)^{2k}k!^2}[F_{2k}(\hat{W}^{*k}, \hat{W}^{*k})].
		\]
	\end{cor}
	
	\section{Pontryagin forms of even and odd curvatures tensors}
	In this section, we will first recall the definition of Hervik et al \cite{Hervik2013minimaltensors} of purely electric and magnetic algebraic curvature tensors. Their approach is to consider the parity of an algebraic curvature tensor under the reflection of a timelike unit vector. In Section \ref{sect: main result}, we will show that for a $4k$-dimensional manifold, the products of Pontryagin classes that land $H_{dR}^{4k}(M)$ vanish if the manifold has a PE or PM Weyl curvature tensor (Theorem \ref{thm: vanishing Pontryagin forms mfds}). We actually prove something stronger, namely that under these assumptions, not only these products of Pontryagin classes vanish but in fact the products of the Pontryagin forms as in Definition \ref{def: Pontrayagin form} vanish everywhere on the manifold (Theorem \ref{thm: vanishing Pontryagin forms vector spaces}). It is not necessary to restrict ourselves to considering reflections in a timelike unit vector or to manifolds with a Lorentzian signature. Accordingly, in Definition \ref{def: even/odd curvature tensor} we introduce the more general notion of algebraic curvature tensors being even or odd with respect to an endomorphism of the vector space that preserves the scalar product $g$ and has determinant $-1$. Our main results, Theorems \ref{thm: vanishing Pontryagin forms vector spaces} and \ref{thm: vanishing Pontryagin forms mfds}, can then be proven for any such even or odd algebraic curvature tensor. In Sections \ref{sect: nonexistence Petrov types} and \ref{sect: nonexistence umbilic hypersurfaces}, we will discuss applications of our main results by providing obstructions to the existence of certain Petrov types on 4-dimensional Lorentzian manifolds and nondegenerate umbilic hypersurfaces of pseudo-Riemannian manifolds.	
	
	\subsection{Motivation: purely electric or magnetic curvature tensors for Lorentzian vector spaces}
	\label{sect: PE and PM general}
	The orthogonal splitting of the space of curvature tensors $\cC(V)$ of a Lorentzian vector space $(V, g)$ introduced by Hervik et al \cite{Hervik2013minimaltensors} is based on the reflection in the orthogonal complement of a timelike unit vector $u \in V$. Concretely, given a timelike unit vector $u \in V$, we define the reflection $\theta_u$ by $\theta_u(u) = -u$ and $\restr{\theta_u}{u^{\bot}} = \id_{u^{\bot}}$. For future reference, we record the following.
	
	\begin{lem}
		\label{lem: reflection in u is sa involution with det -1}
		$\theta_u \in O(V, g)$ and $\det(\theta_u) =-1$.
	\end{lem}
	
	Such an endomorphism naturally acts on algebraic curvature tensors via
	\begin{equation}
		\label{eq: action of endos on curvature tensors}
		\theta_u^*C(w,x,y,z ) = C(	\theta_uw, 	\theta_ux, 	\theta_uy, 	\theta_uz).
	\end{equation}
	It is clear that the $(0, 4)$-tensor $\theta^*_uC$ is again an algebraic curvature tensor and that the map $\theta_u^*\colon \cC(V)\rightarrow \cC(V)$ is an involution. Therefore, we can decompose $C$ as follows. 
	
	\begin{defn}\leavevmode
		\label{def: electric/magnetic part, PE/PM}
		\begin{enumerate}
			\item Let $C \in \cC(V)$ be an algebraic curvature tensor and let $C_{\pm} = \frac{C \pm \theta_u^*C}{2}$. Then $C = C_+ + C_-$ is the decomposition of $C$ in $\pm 1$-eigenvectors of $\theta^*_u$. We call $C_+$ the \emph{electric part of $C$ with respect to $u$} and $C_-$ the \emph{magnetic part of $C$ with respect to $u$}.
			\item We say that $C$ is \emph{purely electric (PE) with respect to $u$} if $C = C_+$ ($C_- = 0$) and that $C$ is \emph{purely magnetic (PM) with respect to $u$} if $C = C_-$ ($C_+ = 0$).
			\item We say that $C$ is \emph{PE} or \emph{PM} if there exists a unit timelike vector $u \in V$ such that $C$ is PE or PM with respect to $u$.
		\end{enumerate}
	\end{defn}
	
	\begin{rem}\leavevmode
		\label{rem: PE/PM}
		\begin{enumerate}
			\item \label{part: vanishing components C for PE/PM}
			It is easy to check that an algebraic curvature tensor $C$ is PE with respect to $u$ if and only if 
			\[
			C(u, x , y, z) = 0
			\]
			for all $x, y, z \in u^{\bot}$. Similarly, $C$ is PM with respect to $u$ if and only if 
			\[
			C(u, x, y, u) = C(w, x, y, z) = 0
			\] 
			for all $w, x, y, z \in u^{\bot}$.
			\item \label{part: RPE/RPM implies PE/PM} If an algebraic curvature tensor $C$ is PE or PM with respect to a unit timelike vector $u$, then so is its associated Weyl curvature tensor $W_C$ \cite[Prop.\ 4.2--3]{Hervik2013minimaltensors}, see also Lemma \ref{lem: curvature tensor even/odd implies Weyl curvature tensor is}.
			\item \label{part: restricts to normal definitions for 4 dimensions}
			For Weyl curvature tensors on a 4-dimensional Lorentzian vector space, the conventional definition of the electric and magnetic parts, see Matte \cite[\S 4]{Matte1953nouvellessolutions}, is equivalent to Definition \ref{def: electric/magnetic part, PE/PM} \cite[Rem.\ 3.2]{Hervik2013minimaltensors}.
		\end{enumerate}
	\end{rem}
	
	\begin{ex}
		\label{ex: PE/PM manifolds}
		Let $I\times_f \Sigma$ be the Lorentzian warped product with an open interval $I\subset \R$ as base and an arbitrary Riemannian manifold $(\Sigma, \sigma)$ as fiber with the metric $g = -dt^2 + f(t)^2\sigma$. Then $I\times_f \Sigma$ has a PE Weyl curvature tensor with respect to the unit timelike vector field $\partial_t$. 
		
		More generally, every spacetime $M$ that admits a unit timelike congruence $u \in \fX(M)$ that is shear-free and irrotational has a Weyl curvature tensor that is PE with respect to $u$ \cite[Prop.\ 3.17]{Hervik2013minimaltensors}. 	
		Other examples of spacetimes with a PE or PM Weyl curvature tensor are spacetimes with irrotational perfect fluids \cite{Barnes1973perfectfluids, Barnes1989irrfluidswithPEWeyl, Lozanovski1999irrperffluidwithPMWeyl, Lozanovski1999perffluidsPMWeyl}.
	\end{ex}
	
	\subsection{A vanishing theorem for Pontryagin forms of even/odd curvature tensors}
	\label{sect: main result}
	In this section, we prove in Theorem \ref{thm: vanishing Pontryagin forms vector spaces} that if an algebraic curvature tensor $C$ is $\theta$-even or -odd (in the sense of Definition \ref{def: even/odd curvature tensor}), $\theta$ has determinant $-1$, and the dimension of the vector space is a multiple of 4, then any product of Pontryagin forms of $C$ that lands in the top exterior power of the dual vector space vanishes. This can be proven generally for all algebraic curvature tensors. As a consequence, we derive in Theorem \ref{thm: vanishing Pontryagin forms mfds} the vanishing of products of Pontryagin classes on compact and orientable pseudo-Riemannian manifolds that have PE or PM Weyl curvature tensors. In Sections \ref{sect: nonexistence Petrov types} and \ref{sect: nonexistence umbilic hypersurfaces}, these results are related to the Petrov classification for 4-dimensional Lorentzian manifolds and nondegenerate umbilic hypersurfaces of pseudo-Riemannian manifolds.
	
	For the rest of this section, let $(V, g)$ be a scalar product space of dimension $n$.
	
	\begin{defn}
		\label{def: even/odd curvature tensor}
		Let $C \in \cC(V)$ be an algebraic curvature tensor and let $\theta \in O(V, g)$. We say that $C$ is \emph{$\theta$-even} if 
		\[
		\theta^*C(u,x,y,z) = C(u,x,y,z),
		\]
		for all $u, x, y, z \in V$. Similarly, we say that $C$ is \emph{$\theta$-odd} if 
		\[
		\theta^*C(u,x,y,z) = -C(u,x,y,z),
		\]
		for all $u, x, y, z \in V$.
	\end{defn}
	
	By Lemma \ref{lem: reflection in u is sa involution with det -1}, we see that a PE or PM algebraic curvature tensor $C \in \cC(V)$ is $\theta$-even or -odd with respect to an orientation reversing isometry $\theta$ of $V$, respectively.
	
	Like being PE or PM, being $\theta$-even- or -odd also descends to Weyl curvature tensors (see part \ref{part: RPE/RPM implies PE/PM} of Remark \ref{rem: PE/PM}).
	
	\begin{lem}
		\label{lem: curvature tensor even/odd implies Weyl curvature tensor is}
		Let $C \in \cC(V)$ be an algebraic curvature tensor and let $\theta \in O(V, g)$. If $C$ is $\theta$-even or -odd, then so is its Weyl curvature tensor $W_C$.
	\end{lem}
	\begin{proof}
		The map $\cW\colon \cC(V)\rightarrow \cW(V, g)$ mapping an algebraic curvature tensor $C$ to its Weyl curvature tensor $W_C$ is a linear map. Indeed, by linearity of the trace, both the Ricci curvature $\Ric_C$ and  scalar curvature $S_C$ depend linearly on $C$. Therefore, also the Schouten tensor 
		\[
		P_C = \frac{1}{n-2}\left(\Ric_C-\frac{S_C}{2(n-1)}g\right)
		\]
		depends linearly on $C$. Finally, by bilinearity of the Kulkarni--Nomizu product, it follows that
		\[
		W_C = C - P_C\owedge g
		\]
		depends linearly on $C$. Moreover, the decomposition theorem from Remark \ref{rem: refined curvature decomposition theorem} ensures that the map $\cW\colon \cC(V)\rightarrow \cW(V, g)$ is $O(V, g)$-equivariant. It follows that if $C$ is $\theta$-even or -odd, then 
		\[
		\theta^*W_C = \theta^*\cW(C) = \cW(\theta^*C) = \pm\cW(C) = \pm W_C.\qedhere
		\]
	\end{proof}
	
	The following result is an easy consequence of the definitions.
	\begin{lem}
		\label{lem: (anti)symmetry on curvature operator}
		Let $C \in \cC(V)$ be an algebraic curvature tensor and let $\hat{C}$ be the induced curvature operator. Let $\theta \in O(V, g)$. If $C$ is  $\theta$-even or -odd, then 
		\begin{equation}
			\label{eq: commutating curvature operator with induced map}
			\hat{C}\circ\wedge^2\theta = \pm \wedge^2\theta\circ\hat{C}, 
		\end{equation} 
		where the positive sign corresponds to the $\theta$-even case and the negative sign to the $\theta$-odd case.
	\end{lem}
	
	\begin{lem}
		\label{lem: Pontryagin forms of (anti)symmetric curvature is symmetric}
		Let $C \in \cC(V)$ be an algebraic curvature tensor and let $\theta \in O(V,g)$. If $C$ is $\theta$-even or -odd, then for all $1\leq k\leq \fl{\frac{n}{2}}$ we have
		\[
		\wedge^{4k}\theta^*(\varpi_k(C)) = \varpi_k(C).
		\]
	\end{lem}
	\begin{proof}
		An easy computation using Theorem \ref{thm: expression Pontryagin form} shows that
		\begin{align*}
			\wedge^{4k}\theta^*(\varpi_k(C)) &= \frac{1}{(2\pi)^{2k}k!^2}\wedge^{4k}\theta^*(F_{2k}(\hat{C}^{*k}, \hat{C}^{*k}))\\
			&\overset{\text{Def. } \ref{def: differential forms induced by exterior endomorphisms}}{=}   \frac{1}{(2\pi)^{2k}k!^2}F_{2k}(\hat{C}^{*k}\circ \wedge^{2k}\theta, \hat{C}^{*k}\circ \wedge^{2k}\theta)\\
			&\overset{\eqref{eq: compatibility * back}}{=} \frac{1}{(2\pi)^{2k}k!^2}F_{2k}((\hat{C}\circ \wedge^{2}\theta)^{*k}, (\hat{C}\circ \wedge^{2}\theta)^{*k})\\
			&\overset{\eqref{eq: commutating curvature operator with induced map}}{=} \frac{1}{(2\pi)^{2k}k!^2}F_{2k}((\pm1)^k(\wedge^{2}\theta \circ \hat{C})^{*k}, (\pm1)^k(\wedge^{2}\theta \circ \hat{C})^{*k})\\
			&\overset{\eqref{eq: compatibility * front}}{=} \frac{(\pm1)^{2k}}{(2\pi)^{2k}k!^2}F_{2k}(\wedge^{2k}\theta \circ \hat{C}^{*k}, \wedge^{2k}\theta \circ\hat{C}^{*k})\\
			&=\frac{1}{(2\pi)^{2k}k!^2}F_{2k}(\hat{C}^{*k}, \hat{C}^{*k}) = \varpi_k(C),
		\end{align*}
		where the first equality on the last line follows from the fact that $\wedge^{2k}\theta$ is an isometry for $\langle-,-\rangle_g$ as $\theta$ is an isometry for $g$.
	\end{proof}
	
 	For a multi-index of positive integers $\alpha = (\alpha_1, \ldots, \alpha_l) \in \Z_{>0}^l$ of length $l$, we denote its \emph{absolute value} by 
	\begin{equation}
		\label{eq: absolute value multi-index}
		|\alpha| = \sum_{i = 1}^l\alpha_i.
	\end{equation}
	Note that if $|\alpha| = k$, then the length of $\alpha$ is less than or equal to $k$ and none of the entries of $\alpha$ are greater than $k$.
	
	In Lemma \ref{lem: Pontryagin forms of (anti)symmetric curvature is symmetric}, we showed that if $C$ is a $\theta$-even or -odd algebraic curvature tensor, then the Pontryagin forms $\varpi_k(C)$ of $C$ are fixed by the map $\wedge^{4k}\theta^*$ induced on $\wedge^{4k}V^*$ by $\theta$. This also means that products of Pontryagin forms that land in $\wedge^{4l}V^*$ for some integer $l$ are fixed by the map $\wedge^{4l}\theta^*$ induced on $\wedge^{4l}V^*$ by $\theta$. If $V$ is $n$-dimensional, then $\wedge^nV^*$ is 1-dimensional and the induced map $\wedge^n\theta^*$ is nothing but multiplication by $\det(\theta^*) = \det(\theta)$ \cite[Ch.\ 7.2]{Greub1978multilinalg}. Finally, if $\theta \in O(V,g)$, we have either $\det(\theta) = \pm 1$. So if $\det(\theta) = -1$, we obtain the following result on the vanishing of (products of) Pontryagin forms.
	
	\begin{thm}[Vanishing theorem Pontryagin forms]
		\label{thm: vanishing Pontryagin forms vector spaces}
		Let $V$ be $4k$-dimensional and let $\theta\in O(V,g)$ with $\det(\theta) = -1$. Suppose that $C$ is an algebraic curvature tensor that is $\theta$-even or -odd. Then for all multi-indices $\alpha$ with $|\alpha| = k$:
		\begin{equation}
			\label{eq: vanishing Pontryagin forms}
			\varpi_{\alpha}(C):= \varpi_{\alpha_1}(C)\wedge\cdots\wedge\varpi_{\alpha_l}(C) = 0 \in \wedge^{4k}V^*.
		\end{equation}
	\end{thm}
	\begin{proof}
		Since $\varpi_{\alpha}(C) \in \wedge^{4k}V^*$, which is 1-dimensional as $V$ is $4k$-dimensional. It follows that $\wedge^{4k}\theta^*$ is nothing but multiplication with $\det(\theta^*) = -1$. So on the one hand, we have
		\begin{equation}
			\label{eq: product of Pontryagin forms is antisymmetric}
			\wedge^{4k}\theta^*(\varpi_{\alpha}(C)) = \det(\theta^*)\varpi_{\alpha}(C) = -\varpi_{\alpha}(C).
		\end{equation}
		
		On the other hand, by Lemma \ref{lem: Pontryagin forms of (anti)symmetric curvature is symmetric},
		\begin{equation}
			\label{eq: product of Pontryagin forms is symmetric}
			\wedge^{4k}\theta^*(\varpi_{\alpha}(C)) =  \wedge^{4\alpha_1}\theta^*(\varpi_{\alpha_1}(C))\wedge\cdots\wedge(\wedge^{4\alpha_l}\theta^*(\varpi_{\alpha_l}(C))) = \varpi_{\alpha}(C).
		\end{equation}
		The result follows by combining \eqref{eq: product of Pontryagin forms is antisymmetric} and \eqref{eq: product of Pontryagin forms is symmetric}.		
	\end{proof}
	
	Now, if $(M, g)$ is a $4k$-dimensional pseudo-Riemannian manifold and if there is some $\theta \in \End(TM)$ satisfying the conditions of Theorem \ref{thm: vanishing Pontryagin forms vector spaces} at each $p\in M$ for which the Weyl curvature tensor of $M$ is even or odd. Then by applying Theorem \ref{thm: vanishing Pontryagin forms vector spaces} to the tangent spaces of $M$, we obtain the following result on the vanishing of products of Pontryagin classes.
	
	\begin{thm}[Vanishing theorem Pontryagin classes]
		\label{thm: vanishing Pontryagin forms mfds}
		Let $(M, g)$ be a $4k$-dimensional pseudo-Riemannian manifold and let $\theta \in \End(TM)$ with $\theta_p \in O(T_pM, g_p)$ and $\det(\theta_p) = -1$ for all $p\in M$. Suppose that the Weyl curvature tensor $W$ is $\theta$-even or -odd at each $p\in M$. Then for all multi-indices $\alpha$ with $|\alpha| = k$,
		\begin{equation}
			\label{eq: Pontraygin forms vanish on mfd}
			\varpi_{\alpha}(W) = 0 \in \Omega^{4k}(M).
		\end{equation}
		In particular, the product of Pontryagin classes
		\begin{equation}
			\label{eq: Pontraygin classes vanish on mfd}
		p_{\alpha}(M):= p_{\alpha_1}(M)\smile\cdots\smile p_{\alpha_l}(M) = 0 \in H^{4k}_{dR}(M).
		\end{equation}
	\end{thm}
	\begin{proof}
		The identity \eqref{eq: Pontraygin forms vanish on mfd} follows directly by applying Theorem \ref{thm: vanishing Pontryagin forms vector spaces} to $(T_pM, g_p)$ for each $p\in M$. Using Corollary \ref{cor: Pontryagin class manifold} and \eqref{eq: Pontraygin forms vanish on mfd}, we see that 
		\[
		p_{\alpha_1}(M)\smile\cdots\smile p_{\alpha_l}(M) = [\varpi_{\alpha_1}(W)\wedge\cdots\wedge\varpi_{\alpha_l}(W)] = [0] = 0.\qedhere
		\]
	\end{proof}
	
	\begin{rem} Let us compare how Theorem \ref{thm: vanishing Pontryagin forms mfds} relates to existing results in the literature.
		\label{rem: comparing main theorem to literature}
		\begin{enumerate}
			\item \label{part: results Mantica en Molinari} Mantica and Molinari introduced the notion of a $C$-\emph{compatible vector} as a special case of the more general notion of $C$-\emph{compatible symmetric $(0, 2)$-tensors} which are studied in \cite{Mantica2012Riemanncompatibletensors, Mantica2014Weylcompatibletensors}. If $u \in V$ is not null and $C$-compatible, then
			\[
			C(u, x , y, z) = 0
			\]
			for all $x, y, z \in u^{\bot}$ \cite[Thm.\ 3.4]{Mantica2014Weylcompatibletensors}. An easy computation using the algebraic Bianchi identity for algebraic curvature tensors shows that the converse also holds. So for timelike unit vectors $u\in V$, $C$-compatibility is equivalent to $C$ being PE with respect to $u$ by part \ref{part: vanishing components C for PE/PM} of Remark \ref{rem: PE/PM}. They also derive a theorem relating the existence of $C$-compatible vectors to the vanishing of Pontryagin classes \cite[Thm.\ 3.6]{Mantica2014Weylcompatibletensors}. Their theorem differs from our Theorem \ref{thm: vanishing Pontryagin forms mfds} in that they use the very strong assumption of having, around each point $p\in M$, a local orthonormal frame where all but two frame fields are $Rm$-compatible. Using this assumption, they conclude that the Riemann tensor is \emph{pure}---that is, the Riemann curvature operator $\hat{Rm}$ is diagonal with respect to an orthonormal basis of $\wedge^2T_pM$ induced by an orthonormal basis of $T_pM$. A theorem of Maillot \cite{Maillot1974courburepur} then ensures the vanishing of all Pontryagin forms and classes.
			\item If the manifold $M$ in Theorem \ref{thm: vanishing Pontryagin forms mfds} is 4-dimensional, then the conclusion is that $p_1(M)$ vanishes. For $4$-dimensional Lorentzian manifolds $M$, it was already known that the Pontryagin class $p_1(M)$ vanishes if the manifold admits a Lorentzian metric with a PE or PM Weyl curvature tensor. This follows directly from \cite[p.\ 323]{Zund1966PontryaginnumbersPRgeometry}, but is not mentioned explicitly. Our Theorem \ref{thm: vanishing Pontryagin forms mfds} extends this result to higher-dimensional manifolds, all pseudo-Riemannian signatures and to a broader class of symmetries.
		\end{enumerate}
	\end{rem}
	
	\begin{rem}[Theorem \ref{thm: vanishing Pontryagin forms mfds} for the Riemann tensor $Rm$]
		\label{rem: main theorem for Rm}
		If $(M, g)$ and $\theta$ are as in Theorem \ref{thm: vanishing Pontryagin forms mfds} but now the Riemann curvature tensor $Rm$ is $\theta$-even or -odd at each point $p \in M$, then we can also apply Theorem \ref{thm: vanishing Pontryagin forms vector spaces} pointwise to $Rm$ to draw the same conclusion as in Theorem \ref{thm: vanishing Pontryagin forms mfds}. Alternatively, in this case Theorem \ref{thm: vanishing Pontryagin forms mfds} can also be applied directly as the Weyl curvature tensor $W$ is also $\theta$-even or -odd by Lemma \ref{lem: curvature tensor even/odd implies Weyl curvature tensor is}.
	\end{rem}
	
	In the light of Mantica and Molinari's vanishing result on Pontryagin classes \cite[Thm.\ 3.6]{Mantica2014Weylcompatibletensors} discussed in part \ref{part: results Mantica en Molinari} of Remark \ref{rem: comparing main theorem to literature}, it is natural to ask whether or not our Theorem \ref{thm: vanishing Pontryagin forms mfds} can be sharpened. Example \ref{ex: manifolds with nonvanishing Pontryagin forms and classes} shows that the conclusion of Theorem \ref{thm: vanishing Pontryagin forms mfds} is optimal in the sense that under the stated assumptions, all products of Pontryagin classes of degree strictly less than $4k$ need not vanish. However, it remains open whether the assumptions on $W$ can be weakened.
	
	\begin{ex}[Theorem \ref{thm: vanishing Pontryagin forms mfds} is optimal]
		\label{ex: manifolds with nonvanishing Pontryagin forms and classes}
		If $k=1$ in Theorem \ref{thm: vanishing Pontryagin forms mfds}, then there are no Pontryagin forms of degree less then $4k$. So to give an example of a pseudo-Riemannian manifold of dimension $4k$ satisfying the assumptions of Theorem \ref{thm: vanishing Pontryagin forms mfds} for which all products of Pontryagin forms of degree less than $4k$ do not vanish, we need at least $k\geq 2$. We give an example for any $k\geq 2$.

		Let $k\geq 2$ be a positive integer. We construct a Lorentzian manifold $(M, g)$ of dimension $4k$ that has a PE Weyl curvature tensor and such that for all multi-indices $\alpha$ with $|\alpha| < k$, the class
		\[
		p_\alpha(M) \neq 0 \in H^{4|\alpha|}_{dR}(M).
		\]
		Then in particular, for all multi-indices $\alpha$ with $|\alpha| < k$, the $4|\alpha|$-form
		\[
		\varpi_\alpha(M) \neq 0 \in \Omega^{4|\alpha|}(M).
		\]
		
		First, consider the manifold $\C P^{2k-2}$, which is of real dimension $4k-4$. It can be shown \cite[p.\ 185]{Milnor74charclasses} that for all multi-indices $\alpha$ with $|\alpha| \leq k-1$
		\[
		p_{\alpha}(\C P^{2k-2})\neq 0.
		\] 
		Equip $\C P^{2k-2}$ with an arbitrary Riemannian metric $h$ and use the $4$-dimensional Minkowski space $(\R^{1, 3}, \eta)$ to form the $4k$-dimensional Lorentzian product manifold 
		\[
		(M, g) = (\R^{1, 3}\times \C P^{2k-2}, \eta\oplus h),
		\] 
		which has a PE Weyl curvature tensor by Example \ref{ex: PE/PM manifolds}.
		
		We will show that for all multi-indices $\alpha$ with $|\alpha| \leq k-1$
		\[
		p_{\alpha}(M)\neq 0 \in H^{4|\alpha|}_{dR}(M).
		\] 
		Denote by $p(M)$, $p(\R^{1, 3})$ and $p(\C P^{2k-2})$ the total Pontryagin classes of these manifolds and denote by $\pi_1 \colon M\rightarrow \R^{1, 3}$ and $\pi_2\colon M\rightarrow \C P^{2k-2}$ the projection maps. Using standard properties of Pontryagin classes (see, for example, \cite[\S 15]{Milnor74charclasses}), we see that
		\begin{align}
			p(M) &= p(\pi_1^*T\R^{1,3}\oplus \pi^*_2T\C P^{2k-2}) \notag\\
			\label{eq: pi_2 relates Pontryagin classes}&= \pi_1^*p(\R^{1,3})\smile \pi^*_2p(\C P^{2k-2}) = \pi^*_2p(\C P^{2k-2}),
		\end{align}
		where the last equality follows as $p(\R^{1, 3}) = 1$ as $\R^{1, 3}$ has a trivial tangent bundle. Since the projection $\pi_2$ is a homotopy equivalence, the induced map $\pi_2^*\colon H^{\bullet}_{dR}(\C P^{2k-2}) \xrightarrow{\sim} H^{\bullet}_{dR}(M)$ is an algebra isomorphism. It follows that for every multi-index $\alpha$ with $|\alpha|\leq k-1$, 
		\[
		p_\alpha(M) = \pi_2^*(p_\alpha(\C P^{2k-2})) \neq 0,
		\]
		as $p_\alpha(\C P^{2k-2}) \neq 0$ \cite[p.\ 185]{Milnor74charclasses} and $\pi_2^*$ is injective. We conclude that for all multi-indices $\alpha$ with $|\alpha| < k$, the form
		\[
		\varpi_\alpha(M) \neq 0 \in \Omega^{4|\alpha|}(M).
		\]
	\end{ex}
	
	As a corollary to Theorem \ref{thm: vanishing Pontryagin forms mfds} we find a new obstruction result for PE or PM metrics for compact orientable manifolds. 
	
	\begin{cor}[Pontryagin class obstruction for PE/PM metrics]
		\label{cor: nonexistence PE/PM}
		Let $M$ be a compact and orientable $4k$-dimensional manifold. Suppose that there exists a multi-index $\alpha$ with $|\alpha| = k$ such that 
		\begin{equation}
			%\label{eq: Pontraygin classes vanish on mfd1}
			p_{\alpha}(M) \neq 0 \in H^{4k}_{dR}(M).
		\end{equation}
		Then $M$ does not admit a pseudo-Riemannian metric such that the Weyl curvature tensor is $PE$ or $PM$ at all points $p \in M$.
	\end{cor}	
	
	\begin{rem} We present some comments on Corollary \ref{cor: nonexistence PE/PM}.
		\begin{enumerate}
			\item Corollary \ref{cor: nonexistence PE/PM} is only nontrivial for compact orientable manifolds $M$. If an $n$-dimensional manifold $M$ is either noncompact or nonorientable, then always $H^{n}_{dR}(M) = 0$ (see for example \cite[Prop.\ 5.IX]{Greub1972Connectionsetc}).
			\item The obstruction in the Pontryagin classes presented in Corollary \ref{cor: nonexistence PE/PM} does not distinguish between PE or PM Weyl curvature tensors. It would be interesting to find obstructions that can distinguish between the (non)existence of PE and PM Weyl curvature tensors separately. 
		\end{enumerate}
	\end{rem}

	We conclude this section by constructing an example of a 4-dimensional compact orientable manifold that admits Lorentzian metrics, but no such has a PE or PM Weyl curvature tensor. Recall that a compact manifold $M$ admits a Lorentzian metric if and only if its Euler characteristic, $\chi(M)$, vanishes (see e.g.\ \cite[Prop.\ 5.37]{ONeill83SemiRiemannian}). In dimension 4, our goal therefore amounts to constructing a manifold $M$ such that
	\[
	\chi(M) = 0 \quad\text{and}\quad p_1(M) \neq 0.
	\]
	
	To determine the nonvanishing of $p_1(M)$, we need a few tools from differential topology. Recall that for a compact orientable 4-dimensional manifold, integration over $M$ determines a linear isomorphism $\int_M\colon H^4_{dR}(M)\rightarrow \R$ and that under this isomorphism, the cup product on $H^2_{dR}(M)$ can be identified with a symmetric bilinear form $B$ on $H^2_{dR}(M)$, which is nondegenerate by Poincaré duality for de Rham cohomology \cite[Ch.\ 5, \S 4, 5]{Greub1972Connectionsetc}. It follows that the matrix of $B$ is diagonalizable with nonzero eigenvalues. Denote by $n_+$ the number of positive eigenvalues of $B$ and by $n_-$ the number of negative eigenvalues of $B$. Then the \emph{topological signature} of $M$ is defined as the integer $\sigma(M) = n_+ - n_-$. In dimension 4, the Hirzebruch signature theorem \cite[Thm.\ 8.2.2]{Hirzebruch1995topomethodsinalggeo} then states that
	\begin{equation}
		\label{eqL: Hirzebruch signature theorem dim 4}
		\sigma(M) = \frac{1}{3}\int_Mp_1(M),
	\end{equation} 
	and therefore $\sigma(M)\neq 0$ if and only if $p_1(M)\neq 0$ as $\int_M$ is an isomorphism.
	
	Finally, if $M_1$ and $M_2$ are 4-dimensional manifolds, denote their connected sum by $M_1 \# M_2$. It can be shown that
	\begin{align}
		\label{eq: identity Euler chars} \chi(M_1 \# M_2) &= \chi(M_1) + \chi(M_2) - 2,\\
		\label{eq: identity topo signatures} \sigma(M_1 \# M_2) &= \sigma(M_1) + \sigma(M_2),
	\end{align}
	see for example \cite[Exc.\ 3.3.6]{Hatcher2002AlgTopo} and \cite[Thm.\ 5.3]{Kirby1989topoof4mfds}, respectively.
	
	\begin{ex}[A compact manifold without PE or PM Lorentzian metrics]
		\label{ex: manifold with nonvanishing p_1}
		Consider the 4-dimensional manifolds $\dT^4$ and $\C P^2$. It can be shown that 
		\begin{align}
			\label{eq: Euler chars} \chi(\dT^4) = 0 &\qquad \chi(\C P^2) = 3\\
			\label{eq: topo signatures} \sigma(\dT^4) = 0 &\qquad \sigma(\C P^2) = 1,
		\end{align}
		where the Euler characteristics follow from the well-known cell-structures of $\dT^4$ and $\C P^2$; $\sigma(\dT^4) = 0$ follows from the fact that $\dT^4$ has a trivial tangent bundle, hence $p_1(\dT^4) = 0$; and for $\sigma(\C P^2)$, see \cite[p.\ 185]{Milnor74charclasses} and use the Hirzebruch signature theorem.	Consider the 4-dimensional manifold $M =\dT^4 \#\dT^4 \# \C P^2 \# \C P^2$. It follows from \eqref{eq: identity Euler chars} and \eqref{eq: Euler chars} that $\chi(M) = 0$ and from \eqref{eq: identity topo signatures} and \eqref{eq: topo signatures} that $\sigma(M) = 2$. So $M$ admits Lorentzian metrics, but none  that have an everywhere PE or PM Weyl curvature tensor.
	\end{ex}

	\subsection{Obstructions to the existence of Petrov types of 4-dimensional Lorentzian manifolds}
	\label{sect: nonexistence Petrov types}
	Corollary \ref{cor: nonexistence PE/PM} formulated an obstruction for the existence of metrics that globally have a PE or PM Weyl curvature tensor in terms of the nonvanishing of products of Pontryagin classes. To conclude this section, we will discuss two applications by proving related obstruction results for certain Petrov types for 4-dimensional Lorentzian manifolds and for foliations by nondegenerate umbilic hypersurfaces.
	
	The Petrov classification, introduced by Petrov \cite[Ch.\ 3]{Petrov1969Einsteinspaces}, is an algebraic classification of Weyl curvature tensors of 4-dimensional Lorentzian manifolds into several types. There are different equivalent approaches to the Petrov classification, also by Bel \cite{Bel1962radiationstates}, Debever \cite{Debever1958superenergie}, Penrose \cite{Penrose1960spinorapproachtoGR}, Pirani \cite{Pirani19578invformulationofgravrad} and others (see Batista \cite[Ch.\ 2]{Batista2013GeneralizationsPetrov} for a survey of 6 different approaches). We follow the approach by Thorpe \cite{Thorpe69Curvature}, which makes use of the curvature operator introduced in Definition \ref{def: curvature operator}.
	
	Let $(V, g)$ be a 4-dimensional Lorentzian vector space. Choose an orientation-defining unit vector $\omega \in \wedge^4V$. Let $\star \in \End(\wedge^2V)$ denote the Hodge-star, which is uniquely defined \cite[Prop.\ 5.2.2]{ONeill95Kerr} by 
	\[
	\xi\wedge\star\eta = -\langle \xi, \eta\rangle_g\omega
	\]
	for all $\xi, \eta \in \wedge^2V$. It is easy to show that $\star$ is self-adjoint and $\star^2 = -1$ \cite[Lem.\ 5.2.3]{ONeill95Kerr}. This means we can view $\wedge^2V$ as a complex vector space of dimension 3 by declaring $i:=\star$. This construction is special for 4-dimensional Lorentzian vector spaces. It is crucial that $V$ is 4-dimensional for $\star$ to be an endomorphism of $\wedge^2V$ and that $V$ is Lorentzian for $\star$ to square to $-1$.
	
	Let $W \in \cW(V, g)$ be a Weyl curvature tensor. That $W$ is trace-free with respect to $g$ implies that $[\hat{W}, \star] =0$ \cite[Cor.\ 5.3.3]{ONeill95Kerr} and, therefore, the curvature operator $\hat{W}$ is $\C$-linear. Moreover, by combining the trace-freeness of $W$ with the first Bianchi-identity, it follows that $\hat{W}$ is trace-free as $\C$-linear operator \cite[Rem.\ 5.4.1]{ONeill95Kerr}. The Petrov type of $W$ is then determined by the Jordan normal form of the curvature operator $\hat{W}$ on $\wedge^2V$. We distinguish the following different Petrov types \cite{Thorpe69Curvature}: 
	\begin{enumerate}
		\item type $O$: $\hat{W} = 0$;
		\item type $I$: $\hat{W}$ is diagonalizable with 3 different complex eigenvalues;
		\item type $D$: $\hat{W}$ is diagonalizable with 2 different complex eigenvalues; 
		\item type $II$: $\hat{W}$ is has 2 Jordan blocks with 2 different complex eigenvalues; 
		\item type $N$: $\hat{W}$ is has 2 Jordan blocks with the same complex eigenvalue (necessarily 0); and
		\item type $III$: $\hat{W}$ is has 1 Jordan block (necessarily with 0 as eigenvalue).
	\end{enumerate}
	Note that this list is exhaustive as there are at most 3 Jordan blocks, in which case $\hat{W}$ is diagonalizable, and if $\hat{W}$ has one complex eigenvalue with multiplicity 3, then this eigenvalue is necessarily 0 as $\hat{W}$ is trace-free as $\C$-linear operator.
	
	Many of the Petrov types imply the vanishing of the Pontryagin form $\varpi_1(W)$. Our results in Section \ref{sect: main result} now allow us to include certain subtypes of type $I$ to this list.
	
	\begin{thm}
		\label{thm: vanishing of Pontryagin form for Petrov types}
		Let $W$ be a Weyl curvature tensor on a 4-dimensional Lorentzian vector space $(V, g)$. Then $\varpi_1(W) = 0$, if $W$ has either of the following Petrov types:
		\begin{enumerate}
			\item type $O$,
			\item type $I$ or $D$ with real or imaginary eigenvalues,
			\item type $II$ with real or imaginary eigenvalues,
			\item type $N$,
			\item type $III$.
		\end{enumerate}
	\end{thm}
	\begin{proof}
		For type $O$, $N$, $III$ and type $D$ and $II$ with real or imaginary eigenvalues, this was proven by Avez \cite[Cor.\ 4]{Avez70charclassesandWeyl}, Zund \cite[Thm.\ 2]{Zund1968topoaspectsBelPetrov} and Porter and Thompson \cite[Thm.\ 3--5]{Porter71topoinGR}. 
		
		If $W$ has type $I$ with real or imaginary eigenvalues, then it follows that $W$ is PE or PM, respectively, see for example \cite{Wylleman2006Completeclassification}, \cite[Rem.\ 3.8]{Hervik2013minimaltensors} or \cite[p.\ 1556]{McIntosh1994PEPMWeyl}. So $\varpi_1(W)=0$ by Theorem \ref{thm: vanishing Pontryagin forms vector spaces}.
	\end{proof}
	
	\begin{cor}[Pontryagin class obstruction for Petrov types]
		\label{cor: nonexistence Petrov types}
		Let $W$ be the Weyl curvature tensor of a 4-dimensional Lorentzian manifold $(M, g)$. Then $p_1(M) = 0$, if $W$ has either of the following Petrov types globally:
		\begin{enumerate}
			\item type $O$,
			\item type $I$ or $D$ with real or imaginary eigenvalues,
			\item type $II$ with real or imaginary eigenvalues,
			\item type $N$,
			\item type $III$.
		\end{enumerate}
				
		Conversely, if $M$ is a compact orientable 4-dimensional manifold such that $p_1(M)\neq 0$, then $M$ does not admit Lorentzian metrics that are globally of the Petrov types listed above. \qed
	\end{cor}
	
	For example, the manifold constructed in Example \ref{ex: manifold with nonvanishing p_1} does not admit Lorentzian metrics with a Weyl curvature tensor that has any of the Petrov types listed in Corollary \ref{cor: nonexistence Petrov types} globally. 
	
	\subsection{Obstructions to the existence of foliations by nondegenerate umbilic hypersurfaces of pseudo-Riemannian manifolds}
	\label{sect: nonexistence umbilic hypersurfaces}
	Let $(M, g)$ be a pseudo-Riemannian manifold and consider a hypersurface $\Sigma \hookrightarrow M$. Denote $\sigma = \restr{g}{\Sigma}$ and assume that $\sigma$ is a pseudo-Riemannian metric. Then the tangent bundle of $M$ over $\Sigma$ splits as 
	\begin{equation}
		\label{eq: split tangent bundle}
		\restr{TM}{\Sigma} = T\Sigma\oplus \cN\Sigma,
	\end{equation}
	where $\cN\Sigma$ denotes the normal bundle of $\Sigma$ in $M$. Using the splitting \eqref{eq: split tangent bundle}, we can decompose the Levi-Civita connection of $M$ into its tangential and normal components at $\Sigma$. More precisely, let $X, Y\in \fX(\Sigma)$ and extend $X$ and $Y$ arbitrarily to vector fields $X$ and $Y$ of an open neighbourhood of $\Sigma$ in $M$. Then 
	\begin{equation}
		\label{eq: split Levi-Civita connection}
		\restr{\nabla^M_XY}{\Sigma} = \nabla^\Sigma_XY + \II(X, Y),
	\end{equation}
	where the first term on the right-hand side of \eqref{eq: split Levi-Civita connection} is the Levi-Civita connection of $(\Sigma, \sigma)$ and the second term is the \emph{second fundamental form of the embedding}.  As $\cN\Sigma$ is a line bundle, we can locally fix a unit normal vector $N$ to $\Sigma$ generating $\cN\Sigma$ and therefore the second fundamental form takes the shape 
	\begin{equation}
		\label{eq: def scalar 2nd fundamental form}
		\II(X, Y) = h(X, Y)N,
	\end{equation}
	for all $X,Y \in \fX(\Sigma)$, where $h$ is a symmetric $(0, 2)$-tensor field on $\Sigma$ \cite[Prop.\ 8.1]{Lee18Riemannmfds}.
	
	\begin{defn}
		A nondegenerate hypersurface $\Sigma \hookrightarrow M$ is \emph{umbilic} if around each point $p\in \Sigma$ there is a neighbourhood $U\subseteq \Sigma$ with a fixed unit normal vector field $N$ such that the scalar second fundamental form $h$ is of the form $h = f\sigma$ for some smooth function $f \in C^{\infty}(U)$. If on any such neighbourhood $f \equiv 0$, then $\Sigma$ is \emph{totally geodesic}.
	\end{defn}
	
	Umbilic hypersurfaces are studied both in the setting of Riemannian geometry (\cite{Cairns1990Totallyumbilicfoliations, Gomes2008Umbilicalfoliations, Langevin2008Conformalgeometryfoliations} are just a few examples) and in the setting of Lorentzian geometry, where umbilic spacelike hypersurfaces may appear as time slices of spacetimes \cite{Avalos2025energypositivityspacetimesexpanding, Brasil2025spacelikefoliations,Ferrando1992inhomogeneousspaces}. On an umbilic nondegenerate hypersurface, the Weyl curvature tensor was shown in \cite[\S 2]{Sharma2023conformalflatness} to satisfy
	\begin{equation}
		\label{eq: W is PE wrt N}
		W(N, X, Y, Z) = 0
	\end{equation}
	for all $X, Y, Z \in \fX(\Sigma)$. The proof given in \cite{Sharma2023conformalflatness} is a long and direct computation. There is a more elegant way to see this, which we present in the following theorem. Note that by similar reasoning as in part \ref{part: vanishing components C for PE/PM} of Remark \ref{rem: PE/PM}, \eqref{eq: W is PE wrt N} is equivalent to $W$ being even under $\theta \in \End(\restr{TM}{\Sigma})$ defined by $\theta(N) = -N$ and $\restr{\theta}{T\Sigma} = \id_{T\Sigma}$, which is a globally well-defined vector bundle morphism by \eqref{eq: split tangent bundle} even though $N$ is defined only locally.
		
	\begin{thm}
		\label{thm: Weyl curvature tensor PE wrt unit normal vector field}
		Suppose that $\Sigma \hookrightarrow M$ is a nondegenerate umbilic hypersurface. Then the Weyl curvature tensor $W$ of $M$ is $\theta$-even on $\Sigma$.
	\end{thm}
	\begin{proof}
		Let $p \in \Sigma$ be an arbitrary point. The question of whether or not $W$ is $\theta$-even at $p$, depends only pointwise on $W$. Therefore, without loss of generality, we may assume (after possibly passing to a neighbourhood of $p$) that a unit normal vector field $N$ of $\Sigma$ exists globally on $\Sigma$ and therefore $\cN\Sigma \cong \Sigma \times \R$.
		
		By the discussion above, we need to show that 
		\[
		W(N, X, Y, Z) = 0
		\]
		for all $X, Y, Z \in \fX(\Sigma)$ and for the fixed unit normal vector field $N$. Recall that the Codazzi equation reads
		\begin{equation}
			\label{eq: Codazzi}
			Rm(X, Y, Z, N) = g(N, N)(Dh)(Z, X, Y),
		\end{equation}
		where $(Dh)(Z, X, Y) = -(\nabla h)(Z, X, Y) + (\nabla h)(Z, Y, X)$ is the \emph{exterior covariant derivative} of $h$ \cite[p.\ 236]{Lee18Riemannmfds}. Since $\Sigma$ is an umbilic hypersurface in $M$ with a global unit normal vector field $N$, there exists a smooth function $\phi \in C^{\infty}(M)$ such that $\Sigma$ is totally geodesic with respect to the conformally equivalent metric $\tilde{g} = e^{2\phi}g$ and $\tilde{h}\equiv 0$, where $\tilde{h}$ is the second fundamental form of $\Sigma$ with respect to $\tilde{g}$. This is shown, for example, in \cite{curry2015introductionconformalgeometrytractor} using tractor calculus, but a more elementary proof of this statement is given in Lemma \ref{lem: umbilic is conformally totally geodesic} after the proof of this theorem.
		
		Let $\tilde{N} = e^{-\phi}N$ be the rescaled unit normal vector field to $\Sigma$ with respect to $\tilde{g}$ and let $\widetilde{Rm}$ and $\widetilde{W}$ denote the Riemann and Weyl curvature tensors of $\tilde{g}$. By \eqref{eq: Codazzi}, we obtain
		\[
		\widetilde{Rm}(X, Y, Z, \tilde{N}) = \tilde{g}(\tilde{N}, \tilde{N})(D\tilde{h})(Z, X, Y) = 0.
		\]
		As $\tilde{N}$ and $N$ are scalar multiples of one another, it follows that $\widetilde{Rm}$ is $\theta$-even on $\Sigma$ by similar reasoning as in part \ref{part: vanishing components C for PE/PM} of Remark \ref{rem: PE/PM}. Hence Lemma \ref{lem: curvature tensor even/odd implies Weyl curvature tensor is} implies that $\widetilde{W}$ is also $\theta$-even on $\Sigma$. Since the Weyl curvature tensor is conformally invariant, we conclude that $W$ is $\theta$-even on $\Sigma$.
	\end{proof}
	
	\begin{lem}
		\label{lem: umbilic is conformally totally geodesic}
		Suppose that $\Sigma \hookrightarrow M$ is a nondegenerate umbilic hypersurface with a global unit normal vector field $N$. Then there exists a smooth function $\phi \in C^{\infty}(M)$ such that $\Sigma$ is totally geodesic with respect to the conformally equivalent metric $\tilde{g} = e^{2\phi}g$.
	\end{lem}
	\begin{proof}
		Let $\phi \in C^\infty(M)$ be an arbitrary function and $\tilde{g} = e^{2\phi}g$ the corresponding conformally equivalent metric. Let $\tilde{N} = e^{-\phi}N$ be the rescaled unit normal vector field to $\Sigma$ with respect to $\tilde{g}$. We will first relate the scalar second fundamental forms $h$ of $g$ and $\tilde{h}$ of $\tilde{g}$.
		Note that the Levi-Civita connection $\nabla$ of the metric $g$ and the Levi-Civita connection $\widetilde{\nabla}$ of the metric $\tilde{g}$ are related by \cite[Prop.\ 7.29]{Lee18Riemannmfds}
		\begin{equation}
			\label{eq: relation LC connections}
			\widetilde{\nabla}_XY = \nabla_XY + Y(\phi)X+X(\phi)Y-g(X, Y)\nabla\phi
		\end{equation}
		for all $X, Y \in \fX(M)$. By using \eqref{eq: split Levi-Civita connection} and \eqref{eq: def scalar 2nd fundamental form} twice on \eqref{eq: relation LC connections} and comparing normal components, we find that the scalar second fundamental forms of $\Sigma$ with respect to the metrics $\tilde{g}$ and $g$ are related by 
		\begin{equation}
			\label{eq: relation scalar second fundamental forms 1}
			e^{-\phi}\tilde{h}(X, Y) = h(X, Y)-g(N, N)g(\nabla\phi, N)\sigma(X, Y)
		\end{equation}
		for all $X, Y \in \fX(\Sigma)$. Since $\Sigma$ is umbilic with respect $g$, it follows that there is a smooth function $f \in C^{\infty}(\Sigma)$ such that $h = f\sigma$. Substituting into \eqref{eq: relation scalar second fundamental forms 1}, we find that 
		\begin{equation}
			\label{eq: relation scalar second fundamental forms 2}
			e^{-\phi}\tilde{h}(X, Y) = \left(f-g(N,N)g(\nabla\phi, N)\right)\sigma(X, Y).
		\end{equation}
		From \eqref{eq: relation scalar second fundamental forms 2}, it follows that it is left to show that for all $f \in C^\infty(\Sigma)$, we can find a $\phi \in C^\infty(M)$ such that 
		\begin{equation}
			\label{eq: totally geodesic requirement}
			g(\nabla\phi, N) = g(N,N)f 
		\end{equation}
		on $\Sigma$. 
		
		Indeed, let $U \subseteq M$ be a tubular neighbourhood of $\Sigma$ with domain $\Sigma\times \{0\} \subset \Sigma_0 \subset \Sigma \times \R$ such that the exponential map
		\[
		E\colon \Sigma_0\rightarrow U\qquad (p, t)\mapsto \exp_p(tN)
		\]
		is a diffeomorphism. Note that using this tubular neighbourhood, $N = \restr{\partial_t}{\Sigma}$.
		Consider the function $\phi_0\colon U\rightarrow \R$ defined by 
		\[
		\phi_0(E(p, t)) = g(N,N)f(p)t.
		\]
		We see that 
		\begin{equation*}
			g(\nabla\phi_0, N) = \restr{g(\nabla\phi_0, \partial_t)}{\Sigma} = \restr{\partial_t(\phi_0)}{\Sigma} = g(N, N)f
		\end{equation*}
		and therefore $\phi_0$ satisfies \eqref{eq: totally geodesic requirement}. So we obtain $\phi$ by extending $\phi_0$ to a function on all of $M$ by using a partition of unity that does not change $\phi_0$ in a neighbourhood of $\Sigma$.
	\end{proof}
	
	Combining Theorems \ref{thm: vanishing Pontryagin forms vector spaces} and \ref{thm: Weyl curvature tensor PE wrt unit normal vector field}, we find that for all multi-indices $\alpha$ with $|\alpha| = k$, the $4k$-form $\varpi_{\alpha}(W)$ vanishes at a nondegenerate umbilic hypersurface of a $4k$-dimensional pseudo-Riemannian manifold.
	
	\begin{thm}[Pontryagin forms vanish at nondegenerate umbilic hypersurfaces]
		\label{thm: products of Pontryagin forms vanish at umbilic hypersurfaces}
		Let $(M, g)$ be a $4k$-dimensional pseudo-Riemannian manifold. Let $\Sigma \hookrightarrow M$ be a nondegenerate umbilic hypersurface. Then for all multi-indices $\alpha$ with $|\alpha| = k$, the $4k$-form $\varpi_{\alpha}(W)$ vanishes at $\Sigma \subseteq M$. \qed
	\end{thm}
	
	Now suppose that a pseudo-Riemannian manifold $(M,g)$ is foliated by nondegenerate umbilic hypersurfaces. Then every point of $M$ lies in a unique leaf of the foliation, which is a nondegenerate umbilic hypersurface. So Theorem \ref{thm: products of Pontryagin forms vanish at umbilic hypersurfaces} yields that for all multi-indices $\alpha$ with $|\alpha| = k$, the $4k$-form $\varpi_{\alpha}(W)$ vanishes everywhere on $M$. This gives an obstruction in the Pontryagin classes to the existence of nondegenerate umbilic foliations by hypersurfaces. To the best of the author's knowledge, this is the first of such an obstruction using Pontryagin classes. However, some nonexistence results are known. In \cite{deAlmeida2017umbilicfolswithintegrablenormalbundle} it is shown that odd-dimensional spheres do not admit umbilic foliations with an integrable normal bundle. More generally, in \cite{Langevin2008Conformalgeometryfoliations} some nonexistence results are given for foliations by umbilic hypersurfaces for compact Riemannian manifolds of constant sectional curvature. However, both \cite{deAlmeida2017umbilicfolswithintegrablenormalbundle} and \cite{Langevin2008Conformalgeometryfoliations} do not use methods based on characteristic classes. We conclude the paper with the following obstruction to the existence of foliations by nondegenerate umbilic hypersurfaces.
	
	\begin{thm}[Pontryagin class obstruction for nondegenerate umbilic foliations]
		\label{thm: products of Pontryagin classes for umbilic foliations}
		Let $(M, g)$ be a $4k$-dimensional pseudo-Riemannian manifold. If there exists a foliation of $(M, g)$ by nondegenerate umbilic hypersurfaces. Then for all multi-indices $\alpha$ with $|\alpha| = k$, we have $p_{\alpha}(M) = 0$.
		
		Conversely, if $M$ is compact and orientable and there exists a multi-index $\alpha$ with $|\alpha|=k$ such that $p_{\alpha}(M)\neq 0$, then there exist no pseudo-Riemannian metric on $M$ and codimension 1 foliation of $M$ for which all leaves of the foliation are nondegenerate and umbilic. \qed
	\end{thm}
		
	\appendix
	\section{Proof of Theorem \ref{thm: expression Pontryagin form}}
	\label{app: proof expression Pontryagin forms}
	In this appendix we will provide a proof of Theorem \ref{thm: expression Pontryagin form}. For which we first introduce the following useful notation. Recall that $[n] = \{1, \ldots, n\}$. We denote by $\cP([n])$ the powerset of $[n]$ and by $\cP([n])_k$ the set of subsets of $[n]$ with  $k$ elements. If we consider a multi-index $I = (i_1, \ldots, i_k)\in [n]^k$ and $\sigma \in S_I$ is a permutation of $I$, then we denote by $I_{\sigma} = (\sigma(i_1), \ldots, \sigma(i_k))$ the permuted multi-index. Also, if $T \in (V^*)^{\otimes k}$ is a $k$-multilinear map and $\sigma \in S_k$ is a permutation on $k$ elements. We define $T_\sigma \in(V^*)^{\otimes k}$ by 
	\[
	T_\sigma(x_1,\ldots, x_k) = T(x_{\sigma(1)},\ldots, x_{\sigma(k)}),
	\]
	for all $x_1, \ldots, x_k \in V$.
	
	If $e$ is an orthonormal basis for $V$, then we define the \emph{causal character signs} of the basis vectors by $\eps_i = g(e_i, e_i) \in \{\pm 1\}$. More generally, a multi-index $I = (i_1, \ldots, i_k)\subseteq [n]$ of length $k$ defines a $k$-vector $e_I = e_{i_1}\wedge\cdots\wedge e_{i_k}$. We define the causal character sign of $e_I$ to be $\eps_I = \langle e_I, e_I\rangle_g$. In other words,
	\[
	\eps_I = \prod_{i\in I}\eps_i,
	\]
	if $I$ has no repeating indices, and $\eps_I = 0$ otherwise. Note this does not depend on the order of $(i_1, \ldots, i_k)$, i.e.\ for all $\sigma \in S_I$, we have $\eps_I = \eps_{I_\sigma}$.
	
	Choosing a basis $e = (e_1, \ldots, e_n)$ for $V$ and an algebraic curvature tensor $C\in \cC(V)$, we obtain the 2-forms $\tensor{\Omega}{^i^j}$ defined by 
	\begin{equation}
		\label{eq: 2-forms curvature operator}
		\hat{C}(\xi) = \sum_{i, j\in[n]}\tensor{\Omega}{^i^j}(\xi)e_i\wedge e_j,
	\end{equation}
	which can be chosen in such a way that $\tensor{\Omega}{^i^j} = -\tensor{\Omega}{^j^i}$.  It is easy to show that the 2-forms $\tensor{\Omega}{^i^j}$ of the curvature operator as defined in \eqref{eq: 2-forms curvature operator} and the 2-forms of the curvature matrix as defined in \eqref{eq: 2-forms curvature matrix} are related via 
	\begin{equation}
		\label{eq: relation 2-forms}
		\tensor{\Omega}{^i^j} = \frac{1}{2}\eps_i\tensor{\Omega}{_i^j}\quad \text{(no summation over $i$)}.
	\end{equation}
	
	The outline of the proof is as follows. In Lemma \ref{lem: higher curvature operators}, we first express the higher curvature operators of $C$ in terms of the 2-forms introduced in \eqref{eq: 2-forms curvature operator}. This way, we can also rewrite the $4k$-form $F_{2k}(\hat{C}^{*k}, \hat{C}^{*k})$ from Theorem \ref{thm: expression Pontryagin form} in terms of the 2-forms from \eqref{eq: 2-forms curvature operator}. Then in Lemma \ref{lem: expression sigma polynomials}, we will give an explicit expression for the polynomials $\sigma_k^{(n)}$. Using the obtained formulas, we can give an explicit expression for the Pontryagin form $\varpi_k(C)$, which after some algebraic manipulations is seen to equal the right-hand side of Theorem \ref{thm: expression Pontryagin form}
		
	\begin{lem}
		\label{lem: higher curvature operators}
		For all $1\leq k\leq \fl{\frac{n}{2}}$,
		\begin{equation}
			\label{eq: expression higher curvature operators}
			\hat{C}^{*k} = \sum_{I \in [n]^{2k}}(\tensor{\Omega}{^{i_1}^{i_2}}\wedge\cdots\wedge \tensor{\Omega}{^{i_{2k-1}}^{i_{2k}}})\otimes e_I.
		\end{equation}
	\end{lem}
	\begin{proof}
		This follows directly by induction. The base case ($k=1$) follows by definition of the 2-forms of the curvature operator in \eqref{eq: 2-forms curvature operator} and the induction step is performed by applying \eqref{eq: expansion * product} to $A = \hat{C}^{*k}$ and $B = \hat{C}$.
	\end{proof}
	
	We proceed to give an explicit formula for the polynomials $\sigma^{(n)}_k$.
	
	\begin{lem}
		\label{lem: expression sigma polynomials}
		Let $X = (\tensor{x}{_i^j})_{i, j \in [n]} \in M_n(\R)$ be a square matrix. Then 
		\begin{equation}
			\sigma_{k}^{(n)}(X) = \frac{1}{k!}\sum_{I, J\in [n]^k}\sgn(I;J)\tensor{x}{_{i_1}^{j_1}}\cdots\tensor{x}{_{i_{k}}^{j_{k}}},
		\end{equation}
		where $\sgn(I;J)$ equals $\sgn(\sigma)$ if $I$ and $J$ have no repeated elements and $J = I_\sigma$, and equals 0 otherwise.
	\end{lem}
	\begin{proof}
		By definition of the determinant and the polynomials $\sigma^{(n)}_k$, we have 
		\[
		\det(I + tX) = 1 + \sum_{k \in [n]}t^k\sigma^{(n)}_k(X) = \sum_{\sigma\in S_n}\sgn(\sigma)\prod_{i \in [n]}(\tensor{\delta}{_i^{\sigma(i)}} + t\tensor{x}{_i^{\sigma(i)}}),
		\]
		where $\delta$ denotes the Kronecker delta. Note that
		\begin{equation}
			\label{eq: working equation sigma expression 1}
			\prod_{i\in [n]}(\tensor{\delta}{_i^{\sigma(i)}} + t\tensor{x}{_i^{\sigma(i)}}) = \sum_{I\subseteq [n]}\left(\prod_{i \in I}t\tensor{x}{_i^{\sigma(i)}}\right)\left(\prod_{i \notin I}\tensor{\delta}{_i^{\sigma(i)}}\right),
		\end{equation}
		so a summand in \eqref{eq: working equation sigma expression 1} vanishes unless $\sigma$ fixes $[n]\setminus I$. The resulting expression reads 
		\begin{equation}
			\label{eq: working equation sigma expression 1.5}
			1 + \sum_{k \in [n]}t^k\sigma^{(n)}_k(X) = \sum_{\sigma\in S_n}\sum_{I\subseteq [n]}\sgn(\sigma)\left(\prod_{i \in I}t\tensor{x}{_i^{\sigma(i)}}\right)\left(\prod_{i \notin I}\tensor{\delta}{_i^{\sigma(i)}}\right)
		\end{equation}
		
		We now re-index the sum in the right-hand side of \eqref{eq: working equation sigma expression 1.5} in order to isolate the powers of $t$, keeping in mind that a summand in \eqref{eq: working equation sigma expression 1.5} vanishes unless $\sigma$ fixes $[n]\setminus I$. Consider the sets 
		\[
		A_1 := \{(\sigma, I) \in S_n\times \cP([n]): \text{$\sigma$ fixes $[n]\setminus I$}\} \qquad \text{and} \qquad A_2 = \coprod_{I \subseteq [n]}S_I.
		\]
		On these sets consider the functions $f\colon A_1\rightarrow A_2$ defined by $(\sigma, I)\mapsto (I, \restr{\sigma}{I})$ and $g\colon A_2\rightarrow A_1$ defined by $(I, \tau) \mapsto (\tilde{\tau}, I)$, where $\tilde{\tau}$ is the extension of $\tau$ to $[n]$ by declaring $\tilde{\tau}$ to be the identity on $[n]\setminus I$. It is easy to see that $f$ and $g$ are mutual inverses, and therefore bijections. Using these bijections and the fact that if $\sigma$ fixes $[n]\setminus I$, then $\sgn(\sigma) = \sgn(\restr{\sigma}{I})$, we see that
		\begin{align*}
			1 + \sum_{k \in [n]}t^k\sigma^{(n)}_k(X) &= \sum_{(\sigma, I)\in A_1}\sgn(\sigma)\left(\prod_{i \in I}t\tensor{x}{_i^{\sigma(i)}}\right) = \sum_{(I, \sigma)\in A_2}t^{|I|}\sgn(\sigma)\left(\prod_{i \in I}\tensor{x}{_i^{\sigma(i)}}\right) \\
			&= 1+ \sum_{k \in [n]}t^k\sum_{\substack{(I, \sigma)\in A_2\\ |I| = k}}\sgn(\sigma)\left(\prod_{i \in I}\tensor{x}{_i^{\sigma(i)}}\right).
		\end{align*}
		Comparing coefficients of $t^k$, we see that 
		\begin{align}
			\label{eq: working equation sigma expression 2} \sigma^{(n)}_k(X) &= \sum_{I\in \cP([n])_k}\sum_{\sigma \in S_I}\sgn(\sigma)\left(\prod_{i \in I}\tensor{x}{_i^{\sigma(i)}}\right) \\
			\label{eq: working equation sigma expression 3}& = \frac{1}{k!}\sum_{I, J \in [n]^k}\sgn(I;J)\tensor{x}{_{i_1}^{j_1}}\cdots\tensor{x}{_{i_{k}}^{j_{k}}}.
		\end{align}
		The factor $\frac{1}{k!}$ in \eqref{eq: working equation sigma expression 3} compensates for the fact that for a fixed $I_0\subseteq [n]$ of length $k$, in \eqref{eq: working equation sigma expression 2} only on the second index of the matrix entries, all possible permutations of $I_0$ are being summed over, whereas in \eqref{eq: working equation sigma expression 3} also on the first index of the matrix entries, all $k!$ different permutations of $I_0$ are being summed over.
	\end{proof} 
	
	We now combine our results from Lemmas \ref{lem: higher curvature operators} and \ref{lem: expression sigma polynomials} to prove Theorem \ref{thm: expression Pontryagin form}.
	
	\begin{proof}[Proof of Theorem \ref{thm: expression Pontryagin form}]
		Using Lemma \ref{lem: expression sigma polynomials}, we find that
		\begin{align}
			\varpi_k(C) &= \frac{1}{(2\pi)^{2k}(2k)!}\sum_{I, J \in [n]^{2k}}\sgn(I;J)\tensor{\Omega}{_{i_1}^{j_1}}\wedge\cdots\wedge\tensor{\Omega}{_{i_{2k}}^{j_{2k}}} \notag\\
			\label{eq: working equation expressions 1}&\overset{\eqref{eq: relation 2-forms}}{=}  \frac{1}{\pi^{2k}(2k)!}\sum_{I, J \in [n]^{2k}}\eps_I\sgn(I;J)\tensor{\Omega}{^{i_1}^{j_1}}\wedge\cdots\wedge\tensor{\Omega}{^{i_{2k}}^{j_{2k}}}.
		\end{align}
		Note that the summands in \eqref{eq: working equation expressions 1} are only nonvanishing if the multi-index $J$ is a permutation of the multi-index $I$ and both have no repeating indices, for $\sgn(I;J) = 0$ otherwise. This means $I$ and $J$ define the same subset $A \in \cP([n])_{2k}$ and when $A$ is considered as a multi-index of length $2k$ with the natural increasing order, there exist unique permutations $\sigma, \tau \in S_{A}$ such that $I = A_{\sigma}$ and $J = A_{\tau}$. Conversely, any subset $A\subseteq [n]$ of length $2k$ and permutations $\sigma, \tau \in S_{A}$ define such $I$ and $J$ uniquely. So the sum in \eqref{eq: working equation expressions 1} can instead be taking over such $A$, $\sigma$ and $\tau$. It follows that $\eps_I = \eps_{A}$ and
		\[
		\sgn(I;J) = \sgn(\tau\sigma^{-1}) = \sgn(\tau)\sgn(\sigma).
		\]
		Combining this, we see that
		\begin{align}
			\varpi_k(C) &= \frac{1}{\pi^{2k}(2k)!}\sum_{A\in \cP([n])_{2k}}\eps_{A}\sum_{\sigma, \tau \in S_{A}}\sgn(\sigma)\sgn(\tau)\tensor{\Omega}{^{\sigma(a_1)}^{\tau(a_1)}}\wedge\cdots\wedge\tensor{\Omega}{^{\sigma(a_{2k})}^{\tau(a_{2k})}} \notag\\
			\label{eq: working equation expressions 2}&= \frac{1}{\pi^{2k}(2k)!}\frac{(2k-1)!!}{2^kk!}\sum_{A\in \cP([n])_{2k}}\eps_A\left(\sum_{\sigma \in S_A}\sgn(\sigma)\tensor{\Omega}{^{\sigma(a_1)}^{\sigma(a_2)}}\wedge\cdots\wedge\tensor{\Omega}{^{\sigma(a_{2k-1})}^{\sigma(a_{2k})}}\right)^{\wedge 2}\\
			\label{eq: working equation expressions 3}&= \frac{1}{(2\pi)^{2k}k!^2}\sum_{A\in \cP([n])_{2k}}\eps_A\left(\sum_{\sigma \in S_A}\sgn(\sigma)\tensor{\Omega}{^{\sigma(a_1)}^{\sigma(a_2)}}\wedge\cdots\wedge\tensor{\Omega}{^{\sigma(a_{2k-1})}^{\sigma(a_{2k})}}\right)^{\wedge 2},
		\end{align}
		where \eqref{eq: working equation expressions 2} follows by Chern's formula on generalized Pfaffian functions \cite[Ch.\ 2]{Chern1955charclassesRiemannian}.
		
		On the other hand, if $I$ and $J$ are two multi-indices of elements in $[n]$ of length $2k$, then it follows that 
		\begin{align}
			\langle e_I, e_J \rangle_g &= \det[(g(e_{i_a}, e_{j_b}))_{a, b\in [2k]}] = \det[(\eps_{i_a}\delta_{i_a, j_b})_{a, b\in [2k]}] \notag\\
			\label{eq: working equation expressions 4}& = \prod_{i\in I}\eps_i\cdot\det[(\delta_{i_a, j_b})_{a, b\in [2k]}] = \eps_I\sgn(I;J).
		\end{align}		
		Computing $F_{2k}(\hat{C}^{*k}, \hat{C}^{*k})$ using Lemma \ref{lem: higher curvature operators} and \eqref{eq: working equation expressions 4}, we find that
		\begin{align}
			F_{2k}(\hat{C}^{*k}, \hat{C}^{*k}) &\overset{\text{Lem. }\ref{lem: higher curvature operators}}{=} \frac{1}{(2k)!^2}\sum_{\sigma \in S_{4k}}\sgn(\sigma)\sum_{I, J \in [n]^{2k}}\langle e_I, e_J \rangle_g \notag\\
			&\qquad\qquad\cdot ((\tensor{\Omega}{^{i_1}^{i_2}}\wedge\cdots\wedge \tensor{\Omega}{^{i_{2k-1}}^{i_{2k}}})\otimes(\tensor{\Omega}{^{j_1}^{j_2}}\wedge\cdots\wedge \tensor{\Omega}{^{j_{2k-1}}^{j_{2k}}}))_{\sigma} \notag\\
			&\overset{\eqref{eq: working equation expressions 4}}{=} \frac{1}{(2k)!^2}\sum_{\sigma \in S_{4k}}\sgn(\sigma)\sum_{I, J \in [n]^{2k}}\eps_I\sgn(I;J) \notag\\
			&\qquad\qquad\cdot ((\tensor{\Omega}{^{i_1}^{i_2}}\wedge\cdots\wedge \tensor{\Omega}{^{i_{2k-1}}^{i_{2k}}})\otimes(\tensor{\Omega}{^{j_1}^{j_2}}\wedge\cdots\wedge \tensor{\Omega}{^{j_{2k-1}}^{j_{2k}}}))_{\sigma} \notag\\
			&=\sum_{I, J \in [n]^{2k}}\eps_I\sgn(I;J)\notag\\
			\label{eq: working equation expressions 5}&\qquad\qquad\cdot(\tensor{\Omega}{^{i_1}^{i_2}}\wedge\cdots\wedge \tensor{\Omega}{^{i_{2k-1}}^{i_{2k}}}\wedge\tensor{\Omega}{^{j_1}^{j_2}}\wedge\cdots\wedge \tensor{\Omega}{^{j_{2k-1}}^{j_{2k}}}),
		\end{align}
		where the final equality follows from the definition of the wedge-product on differential forms. Again using the fact that the terms in \eqref{eq: working equation expressions 5} are only nonvanishing if the multi-index $J$ is a permutation of the multi-index $I$ and both do not have a repeating index and that in this case $I$ and $J$ uniquely define a set $A\in \cP([n])_{2k}$ and permutations $\sigma, \tau \in S_A$ such that $I = A_{\sigma}$ and $J= A_{\tau}$, we can rewrite \eqref{eq: working equation expressions 5} as
		\begin{align}
			&F_{2k}(\hat{C}^{*k}, \hat{C}^{*k}) \notag\\
			&\qquad=\sum_{A\in \cP([n])_{2k}}\eps_A\sum_{\sigma, \tau \in S_A}\sgn(\sigma)\sgn(\tau) \notag\\
			&\qquad\qquad\cdot(\tensor{\Omega}{^{\sigma(a_1)}^{\sigma(a_2)}}\wedge\cdots\wedge \tensor{\Omega}{^{\sigma(a_{2k-1})}^{\sigma(a_{2k})}}\wedge\tensor{\Omega}{^{\tau(a_1)}^{\tau(a_2)}}\wedge\cdots\wedge \tensor{\Omega}{^{\tau(a_{2k-1})}^{\tau(a_{2k})}})\notag\\
			\label{eq: working equation expressions 6}&\qquad= \sum_{A\in \cP([n])_{2k}}\eps_A\left(\sum_{\sigma \in S_A}\sgn(\sigma)\tensor{\Omega}{^{\sigma(a_1)}^{\sigma(a_2)}}\wedge\cdots\wedge\tensor{\Omega}{^{\sigma(a_{2k-1})}^{\sigma(a_{2k})}}\right)^{\wedge 2}.
		\end{align}
		The result follows from combining \eqref{eq: working equation expressions 3} and \eqref{eq: working equation expressions 6}.
	\end{proof}
	
	\subsection*{Data Availability}
	Data sharing not applicable to this article as no datasets were generated or analyzed during
	the current study.
	
	\section*{Declarations}
	\subsection*{Conflict of interests}
	The author has no competing interests to declare that are relevant to the content of this article.

	%%%%%% BIBLIOGRAPHY %%%%%%
	\bibliographystyle{plainurl}
	\bibliography{refs.bib}
\end{document}